\documentclass[twoside]{amsart}
\usepackage{amsmath}
\usepackage{amssymb}
\usepackage[cp850]{inputenc}
\usepackage[bookmarksnumbered,plainpages,hypertex]{hyperref}

\input{xy}
 \xyoption{all}
\newtheorem{theorem}{\sc Theorem}[section]
\newtheorem{proposition}[theorem]{\sc Proposition}

\newtheorem{notations}[theorem]{\sc Notations}
\newtheorem{lemma}[theorem]{\sc Lemma}
\newtheorem{corollary}[theorem]{\sc Corollary}
\theoremstyle{definition}
\newtheorem{definition}[theorem]{\sc Definition}
\newtheorem{definitions}[theorem]{\sc Definitions}

\theoremstyle{remark}
\newtheorem{remark}[theorem]{\sc Remark}

\newtheorem{claim}[theorem]{}

\def\ot{\otimes}
\def\cot{\square}

\def\w{\wedge }

\newcommand{\diagMorphIdeal}{\xymatrix{I \ar[rr]^{f} \ar[dr]_{i_I^A}&&J \ar[dl]^{i_J^A} \\&A}}
\newcommand{\diagMultIdeal}{\xymatrix{ 0\ar[r]& IJ \ar[rr]^{i^A_{IJ}} &&  A \ar[rr]^{\pi^A_{I,J}} &&Q_{I,J}\ar[r]&0  \\ && I\ot J  \ar@{.>}[lu]^{\overline{m}_{I,J}}\ar[ur]_{m_{I,J}}}}
\newcommand{\diagComultWedge}{\xymatrix{ 0\ar[r]& C\w_E  D \ar[rr]^{i_{C\w_E  D}^E} &&  E\ar[dr]_{\Delta_{C,D}} \ar[rr]^{p^E_{C\w_E  D}} &&\frac{E}{C\w_E  D}\ar[r]\ar@{.>}[dl]^{\overline{\Delta}_{C,D}}&0  \\
 &&&& \frac{E}{C}\ot\frac{E}{D}}}
 \allowdisplaybreaks

\begin{document}
\title{Associated Graded Algebras and Coalgebras}
\author{Alessandro Ardizzoni}
\address{University of Ferrara, Department of Mathematics, Via Machiavelli
35, Ferrara, I-44121, Italy}
\email{alessandro.ardizzoni@unife.it}
\urladdr{http://www.unife.it/utenti/alessandro.ardizzoni}
\author{Claudia Menini}
\address{University of Ferrara, Department of Mathematics, Via Machiavelli
35, Ferrara, I-44121, Italy}
\email{men@unife.it}
\urladdr{http://www.unife.it/utenti/claudia.menini}
\thanks{This paper was written while the authors were members of
G.N.S.A.G.A. with partial financial support from M.I.U.R. within the National Research Project PRIN 2007.}

\begin{abstract}
We investigate the notion of associated graded coalgebra (algebra) of a
bialgebra with respect to a subbialgebra (quotient bialgebra) and
characterize those which are bialgebras of type one in the framework of
abelian braided monoidal categories.
\end{abstract}

\maketitle
\tableofcontents

\pagestyle{headings}

\section*{Introduction}

Let $H$ be a braided bialgebra in a cocomplete and complete abelian braided
monoidal category $\left( \mathcal{M},c\right) $ satisfying $AB5$. Assume
that the tensor product commutes with direct sums and is two-sided exact.
Let $M$ be in $_{H}^{H}\mathcal{M}_{H}^{H}$. Let $T=T_{H}(M)$ be the
relative tensor algebra and let $T^{c}=T_{H}^{c}(M)$ be the relative
cotensor coalgebra as introduced in \cite{AMS:Cotensor}. Then both $T$ and $%
T^{c}$ have a natural structure of graded braided bialgebra and the natural
algebra morphism from $T$ to $T^{c}$, which coincide with the canonical
injections on $H$ and $M$, is a graded bialgebra homomorphism. Thus its
image is a graded braided bialgebra which is denoted by $H[M]$ and called
\emph{the braided bialgebra of type one associated to }$H$\emph{\ and }$M$.
Ordinary bialgebras of type one were introduced by Nichols in \cite{Ni}.
They came out to play a relevant role in the theory of Hopf Algebras. In
particular, their "coinvariant" part, called Nichols algebra, has been
deeply investigated, see e.g. \cite{Ro}, \cite{AS} and the references
therein.

Let $B\hookrightarrow E$ be a monomorphism in $\mathcal{M}$ which is a
braided bialgebra homomorphism in $\mathcal{M}$. Under some technical
assumptions, we prove in Theorem \ref{teo: typeOne Alg} that the following
assertions are equivalent.

\begin{enumerate}
\item[$\left( 1\right) $] $gr_{B}E$ is the braided bialgebra of type one
associated to $B$ and $\frac{B\wedge _{E}B}{B}$.

\item[$\left( 2\right) $] $gr_{B}E$ is strongly $%
\mathbb{N}
$-graded as an algebra (in the sense of Definition \ref{def: strongly grAlg}%
).

\item[$\left( 3\right) $] $\oplus _{n\in \mathbb{N}}B^{\wedge _{E}n+1}$ is
strongly $%
\mathbb{N}
$-graded as an algebra.

\item[$\left( 4\right) $] $B^{\wedge _{E}n+1}=\left( B\wedge _{E}B\right)
^{\cdot _{E}n}$ for every $n\geq 2.$
\end{enumerate}

Here $gr_{B}E$ denotes the associated graded coalgebra $\oplus _{n\in
\mathbb{N}}\frac{B^{\wedge _{E}n+1}}{B^{\wedge _{E}n}}$. As an application,
in Corollary \ref{coro: pre-Lift}, we consider the case of a subbialgebra $H$
of a bialgebra $E$ over a field $K$.

Similar results are obtained in the case $gr_{I}E:=\oplus _{n\in \mathbb{N}}%
\frac{I^{n}}{I^{n+1}}$ where $I$ is the kernel of an epimorphism $\pi
:E\rightarrow B$ which is a braided bialgebra homomorphism in $\mathcal{M}$.

The paper is organized as follows. In Section \ref{sec: Preliminaries and
Notations} we recall some definitions and introduce basic notations needed
in the paper. In Section \ref{sec: ass Grad Coalg}, we prove that the
associated graded coalgebra $gr_{C}E:=\oplus _{n\in \mathbb{N}}\frac{%
C^{\wedge _{E}n+1}}{C^{\wedge _{E}n}}$ for a given subcoalgebra $C$ of a
coalgebra $E$ in $\mathcal{M}$, is a strongly $%
\mathbb{N}
$-graded coalgebra (see Theorem \ref{teo: gr dual}) and hence it can be
characterized as in Theorem \ref{teo: graded dual}. Dual results are
obtained in Section \ref{sec: ass Grad alg} for the associated graded
algebra $gr_{I}A:=\oplus _{n\in \mathbb{N}}\frac{I^{n}}{I^{n+1}}$, where $I$
is an ideal of an algebra $A$ in $\mathcal{M}$. In Section \ref{sec: ass
Grad Coalg Bialg}, we firstly show that $gr_{B}E$ is a graded braided
bialgebra in $\left( \mathcal{M},c\right) $ (see Theorem \ref{teo: gr_BE
Bialg}) and then prove Theorem \ref{teo: typeOne Alg} which is the main
result of the paper. Section \ref{sec: ass Grad alg Bialg} deals with the
dual results. For the reader's sake, some technicalities are collected in
Appendix \ref{sec: techn}.

Finally we would like to outline that many results are firstly stated and
proved in the coalgebra case where we found the proofs less straightforward.
Also proofs in the algebra case are not given whenever they would have been
an easy adaptation of the coalgebra ones.

In \cite{Connected}, results of the present paper are applied to study
strictly graded bialgebras.

\section{Preliminaries and Notations\label{sec: Preliminaries and Notations}}

\textbf{Notations.} \ Let $[(X,i_{X})]$ be a subobject of an object $E$ in
an abelian category $\mathcal{M},$ where $i_{X}=i_{X}^{E}:X\hookrightarrow E$
is a monomorphism and $[(X,i_{X})]$ is the associated equivalence class. By
abuse of language, we will say that $(X,i_{X})$ is a subobject of $E$ and we
will write $(X,i_{X})=(Y,i_{Y})$ to mean that $(Y,i_{Y})\in \lbrack
(X,i_{X})]$. The same convention applies to cokernels. If $(X,i_{X})$ is a
subobject of $E$ then we will write $(E/X,p_{X})=\text{Coker}(i_{X})$, where
$p_{X}=p_{X}^{E}:E\rightarrow E/X$.\newline
Let $(X_{1},i_{X_{1}}^{Y_{1}})$ be a subobject of $Y_{1}$ and let $%
(X_{2},i_{X_{2}}^{Y_{2}})$ be a subobject of $Y_{2}$. Let $%
x:X_{1}\rightarrow X_{2}$ and $y:Y_{1}\rightarrow Y_{2}$ be morphisms such
that $y\circ i_{X_{1}}^{Y_{1}}=i_{X_{2}}^{Y_{2}}\circ x$. Then there exists
a unique morphism, which we denote by $y/x=\frac{y}{x}:Y_{1}/{X_{1}}%
\rightarrow Y_{2}/{X_{2},}$ such that $\frac{y}{x}\circ
p_{X_{1}}^{Y_{1}}=p_{X_{2}}^{Y_{2}}\circ y$:
\begin{equation*}
\xymatrix{ X_1 \ar[d]_{x} \ar@{^{(}->}[r]^{i_{X_1}^{Y_1}} & Y_1 \ar[d]_{y}
\ar[r]^{p_{X_1}^{Y_1}} & \frac{Y_1}{X_1} \ar[d]^{\frac{y}{x}} \\ X_2
\ar@{^{(}->}[r]^{i_{X_2}^{Y_2}} & Y_2 \ar[r]^{p_{X_2}^{Y_2}} &
\frac{Y_2}{X_2} }
\end{equation*}%
$\delta _{u,v}$ will denote the Kronecker symbol for every $u,v\in
\mathbb{N}
$.

\begin{claim}
\textbf{Monoidal Categories.} Recall that (see \cite[Chap. XI]{Kassel}) a
\emph{monoidal category}\textbf{\ }is a category $\mathcal{M}$ endowed with
an object $\mathbf{1}\in \mathcal{M}$\textbf{\ } (called \emph{unit}), a
functor $\otimes :\mathcal{M}\times \mathcal{M}\rightarrow \mathcal{M}$
(called \emph{tensor product}), and functorial isomorphisms $%
a_{X,Y,Z}:(X\otimes Y)\otimes Z\rightarrow $ $X\otimes (Y\otimes Z)$, $l_{X}:%
\mathbf{1}\otimes X\rightarrow X,$ $r_{X}:X\otimes \mathbf{1}\rightarrow X,$
for every $X,Y,Z$ in $\mathcal{M}$. The functorial morphism $a$ is called
the \emph{associativity constraint }and\emph{\ }satisfies the \emph{Pentagon
Axiom, }that is the following relation
\begin{equation*}
(U\otimes a_{V,W,X})\circ a_{U,V\otimes W,X}\circ (a_{U,V,W}\otimes
X)=a_{U,V,W\otimes X}\circ a_{U\otimes V,W,X}
\end{equation*}%
holds true, for every $U,V,W,X$ in $\mathcal{M}.$ The morphisms $l$ and $r$
are called the \emph{unit constraints} and they obey the \emph{Triangle
Axiom, }that is $(V\otimes l_{W})\circ a_{V,\mathbf{1},W}=r_{V}\otimes W$,
for every $V,W$ in $\mathcal{M}$.

A \emph{braided monoidal category} $(\mathcal{M},c)$ is a monoidal category $%
(\mathcal{M},\otimes ,\mathbf{1})$ equipped with a \emph{braiding} $c$, that
is a natural isomorphism $c_{X,Y}:X\otimes Y\longrightarrow Y\otimes X$ for
every $X,Y,Z$ in $\mathcal{M}$ satisfying
\begin{equation*}
c_{X\otimes Y,Z}=(c_{X,Z}\otimes Y)(X\otimes c_{Y,Z})\qquad \text{and}\qquad
c_{X,Y\otimes Z}=(Y\otimes c_{X,Z})(c_{X,Y}\otimes Z).
\end{equation*}%
For further details on these topics, we refer to \cite[Chapter XIII]{Kassel}.
\end{claim}

It is well known that the Pentagon Axiom completely solves the consistency
problem arising out of the possibility of going from $((U\otimes V)\otimes
W)\otimes X$ to $U\otimes (V\otimes (W\otimes X))$ in two different ways
(see \cite[page 420]{Majid}). This allows the notation $X_{1}\otimes \cdots
\otimes X_{n}$ forgetting the brackets for any object obtained from $%
X_{1},\cdots X_{n}$ using $\otimes $. Also, as a consequence of the
coherence theorem, the constraints take care of themselves and can then be
omitted in any computation involving morphisms in $\mathcal{M}$.\newline
Thus, for sake of simplicity, from now on, we will omit the associativity
constraints.\medskip \newline
The notions of algebra, module over an algebra, coalgebra and comodule over
a coalgebra can be introduced in the general setting of monoidal categories.
Given an algebra $A$ in $\mathcal{M}$ on can define the categories $_{A}%
\mathcal{M}$, $\mathcal{M}_{A}$ and $_{A}\mathcal{M}_{A}$ of left, right and
two-sided modules over $A$ respectively. Similarly, given a coalgebra $C$ in
$\mathcal{M}$, one can define the categories of $C$-comodules $^{C}\mathcal{M%
},\mathcal{M}^{C},{^{C}\mathcal{M}^{C}}$. For more details, the reader is
referred to \cite{AMS}.

\begin{definitions}
\label{abelian assumptions}Let $\mathcal{M}$ be a monoidal category.\newline
We say that $\mathcal{M}$ is an \textbf{abelian monoidal category }if $%
\mathcal{M}$ is abelian and both the functors $X\otimes (-):\mathcal{M}%
\rightarrow \mathcal{M}$ and $(-)\otimes X:\mathcal{M}\rightarrow \mathcal{M}
$ are additive and right exact, for any $X\in \mathcal{M}.$\newline
We say that $\mathcal{M}$ is an \textbf{coabelian monoidal category }if $%
\mathcal{M}^o$ is an abelian monoidal category, where $\mathcal{M}^{o}$
denotes the dual monoidal category of $\mathcal{M}$. Recall that $\mathcal{M}%
^{o}$ and $\mathcal{M}$ have the same objects but $\mathcal{M}^{o}(X,Y)=%
\mathcal{M}(Y,X) $ for any $X,Y$ in $\mathcal{M}$.
\end{definitions}

Given an algebra $A$ in an abelian monoidal category $\mathcal{M}$, there
exist a suitable functor $\otimes _{A}:{_{A}\mathcal{M}_{A}}\times {_{A}%
\mathcal{M}_{A}}\rightarrow {_{A}\mathcal{M}}_{A}$ and constraints that make
the category $({_{A}\mathcal{M}}_{A},\otimes _{A},A)$ abelian monoidal, see
\cite[1.11]{AMS}. The tensor product over $A$ in $\mathcal{M}$ of a right $A$%
-module $V$ and a left $A$-module $W$ is defined to be the coequalizer:
\begin{equation*}
\xymatrix{ (V\otimes A)\otimes W \ar@<.5ex>[rr] \ar@<-.5ex>[rr]&& V\otimes W
\ar[rr]^{_{A}\chi _{V,W}} && V\otimes _{A}W \ar[r] & 0 }
\end{equation*}%
Note that, since $\otimes $ preserves coequalizers, then $V\otimes _{A}W$ is
also an $A$-bimodule, whenever $V$ and $W$ are $A$-bimodules.\medskip
\newline
Dually, let $\mathcal{M}$ be a coabelian monoidal category.\newline
Given a coalgebra $(C,\Delta ,\varepsilon )$ in $\mathcal{M}$, there exist
of a suitable functor $\square _{C}:{^{C}\mathcal{M}^{C}}\times {^{C}%
\mathcal{M}^{C}}\rightarrow {^{C}\mathcal{M}^{C}}$ and constraints that make
the category $({^{C}\mathcal{M}^{C}},\square _{C},C)$ coabelian monoidal.
The cotensor product over $C$ in $\mathcal{M}$ of a right $C$-bicomodule $V$
and a left $C$-comodule $W$ is defined to be the equalizer:
\begin{equation*}
\xymatrix{ 0 \ar[r] & V\cot_{C}W \ar[rr]^{_C\varsigma_{V,W}} && V\otimes W
\ar@<.5ex>[rr] \ar@<-.5ex>[rr]&&V\ot(C\ot W) }
\end{equation*}%
Note that, since $\otimes $ preserves equalizers, then $V\square _{C}W$ is
also a $C$-bicomodule, whenever $V$ and $W$ are $C$-bicomodules.

\begin{claim}
\textbf{Graded Objects.} \label{claim 4.2}Let $\left( X_{n}\right) _{n\in
\mathbb{N}
}$ be a sequence of objects in a monoidal category $\mathcal{M}$ which is
cocomplete abelian and let
\begin{equation*}
X=\bigoplus_{n\in \mathbb{N}}X_{n}
\end{equation*}%
be their coproduct in $\mathcal{M}$. In this case we also say that $X$ is a
\emph{graded object of }$\mathcal{M}$ and that the sequence $\left(
X_{n}\right) _{n\in
\mathbb{N}
}$ defines a graduation on $X.$ A morphism
\begin{equation*}
f:X=\bigoplus_{n\in \mathbb{N}}X_{n}\rightarrow Y=\bigoplus_{n\in \mathbb{N}%
}Y_{n}
\end{equation*}%
is called a \emph{graded homomorphism} whenever there exists a family of
morphisms $\left( f_{n}:X_{n}\rightarrow Y_{n}\right) _{n\in
\mathbb{N}
}$ such that $f=\oplus _{n\in \mathbb{N}}f_{n}$ i.e. such that%
\begin{equation*}
f\circ i_{X_{n}}^{X}=i_{Y_{n}}^{Y}\circ f_{n},\text{ for every }n\in \mathbb{%
N}\text{.}
\end{equation*}%
We fix the following notations:
\begin{equation*}
p_{n}:X\rightarrow X_{n},
\end{equation*}%
be the canonical projection and let
\begin{equation*}
i_{n}:X_{n}\rightarrow X,
\end{equation*}%
be the canonical injection for any $n\in \mathbb{N}$.

Given graded objects $X,Y$ in $\mathcal{M}$ we set%
\begin{equation*}
\left( X\otimes Y\right) _{n}=\oplus _{a+b=n}\left( X_{a}\otimes
Y_{b}\right) .
\end{equation*}%
Then this defines a graduation on $X\otimes Y$ whenever{\Huge \ }the tensor
product commutes with direct sums. We denote by
\begin{equation*}
X_{a}\otimes Y_{b}\overset{\gamma _{a,b}^{X,Y}}{\rightarrow }\left( X\otimes
Y\right) _{a+b}\text{\qquad and\qquad }\left( X\otimes Y\right) _{a+b}%
\overset{\omega _{a,b}^{X,Y}}{\rightarrow }X_{a}\otimes Y_{b}
\end{equation*}%
the canonical injection and projection respectively. We have%
\begin{eqnarray}
\sum\limits_{a+b=n}\left( i_{a}^{X}\otimes i_{b}^{Y}\right) \omega
_{a,b}^{X,Y} &=&\nabla \left[ \left( i_{a}^{X}\otimes i_{b}^{Y}\right)
_{a+b=n}\right]  \label{form nablasum1} \\
\sum\limits_{a+b=n}\gamma _{a,b}^{X,Y}\left( p_{a}^{X}\otimes
p_{b}^{Y}\right) &=&\Delta \left[ \left( p_{a}^{X}\otimes p_{b}^{Y}\right)
_{a+b=n}\right]  \label{form deltasum1}
\end{eqnarray}%
where $\nabla \left[ \left( i_{a}^{X}\otimes i_{b}^{Y}\right) _{a+b=n}\right]
$ denotes the codiagonal morphism associated to the family $\left(
i_{a}^{X}\otimes i_{b}^{Y}\right) _{a+b=n}$ and $\Delta \left[ \left(
p_{a}^{X}\otimes p_{b}^{Y}\right) _{a+b=n}\right] $ denotes the diagonal
morphism associated to the family $\left( p_{a}^{X}\otimes p_{b}^{Y}\right)
_{a+b=n}.$
\end{claim}

\section{The Associated Graded Coalgebra \label{sec: ass Grad Coalg}}

\begin{claim}
\label{def: grCoalg} Let $\mathcal{M}$ be a coabelian monoidal category such
that the tensor product commutes with direct sums.

Recall that a \emph{graded coalgebra} in $\mathcal{M}$ is a coalgebra $%
\left( C,\Delta ,\varepsilon \right) $ where
\begin{equation*}
C=\oplus _{n\in \mathbb{N}}C_{n}
\end{equation*}%
is a graded object of $\mathcal{M}$ such that $\Delta :C\rightarrow C\otimes
C$ is a graded homomorphism i.e. there exists a family $\left( \Delta
_{n}\right) _{_{n\in \mathbb{N}}}$ of morphisms
\begin{equation*}
\Delta _{n}^{C}=\Delta _{n}:C_{n}\rightarrow \left( C\otimes C\right)
_{n}=\oplus _{a+b=n}\left( C_{a}\otimes C_{b}\right) \text{ such that }%
\Delta =\oplus _{n\in \mathbb{N}}\Delta _{n}.
\end{equation*}%
We set
\begin{equation*}
\Delta _{a,b}^{C}=\Delta _{a,b}:=\left( C_{a+b}\overset{\Delta _{a+b}}{%
\rightarrow }\left( C\otimes C\right) _{a+b}\overset{\omega _{a,b}^{C,C}}{%
\rightarrow }C_{a}\otimes C_{b}\right) .
\end{equation*}%
A homomorphism $f:\left( C,\Delta _{C},\varepsilon _{C}\right) \rightarrow
\left( D,\Delta _{D},\varepsilon _{D}\right) $ of coalgebras is a graded
coalgebra homomorphism if it is a graded homomorphism too.
\end{claim}

\begin{definition}
\label{def: strongly grCoalg}Let $(C=\oplus _{n\in
\mathbb{N}
}C_{n},\Delta ,\varepsilon )$ be a graded coalgebra in $\mathcal{M}$. In
analogy with the group graded case (see \cite{NT}), we say that $C$ is a
\emph{strongly }$%
\mathbb{N}
$\emph{-graded coalgebra} whenever

$\Delta _{i,j}^{C}:C_{i+j}\rightarrow C_{i}\otimes C_{j}$ is a monomorphism
for every $i,j\in \mathbb{N},$

where $\Delta _{i,j}^{C}$ is the morphism defined in Definition \ref{def:
grCoalg}.
\end{definition}

\begin{proposition}
\cite[Propositions 2.5 and 2.3]{AM} \label{lem: graded Deltaij}Let $\mathcal{%
M}$ be a coabelian monoidal category such that the tensor product commutes
with direct sums.

1) Let $C=\oplus _{n\in \mathbb{N}}C_{n}$ be a graded object of $\mathcal{M}$
such that there exists a family $\left( \Delta _{a,b}^{C}\right) _{_{a,b\in
\mathbb{N}
}}$
\begin{equation*}
\Delta _{a,b}^{C}:C_{a+b}\rightarrow C_{a}\otimes C_{b},
\end{equation*}%
of morphisms and a morphism $\varepsilon _{0}^{C}:C_{0}\rightarrow \mathbf{1}
$ which satisfy
\begin{gather}
\left( \Delta _{a,b}^{C}\otimes C_{c}\right) \circ \Delta
_{a+b,c}^{C}=\left( C_{a}\otimes \Delta _{b,c}^{C}\right) \circ \Delta
_{a,b+c}^{C}\text{,}  \label{form: locDelta} \\
\left( C_{d}\otimes \varepsilon _{0}^{C}\right) \circ \Delta
_{d,0}^{C}=r_{C_{d}}^{-1},\qquad \left( \varepsilon _{0}^{C}\otimes
C_{d}\right) \circ \Delta _{0,d}^{C}=l_{C_{d}}^{-1}\text{,}
\label{form: locEps}
\end{gather}%
for every $a,b,c\in
\mathbb{N}
$. Then there exists a unique morphism $\Delta _{C}:C\rightarrow C\otimes C$
such that
\begin{equation}
(p_{a}^{C}\otimes p_{b}^{C})\circ \Delta _{C}=\Delta _{a,b}^{C}\circ
p_{a+b}^{C},\text{ for every }a,b\in
\mathbb{N}
\label{form: coro grCoalg1}
\end{equation}%
holds. Moreover $\left( C=\oplus _{n\in \mathbb{N}}C_{n},\Delta
_{C},\varepsilon _{C}=\varepsilon _{0}^{C}p_{0}^{C}\right) $ is a graded
coalgebra.

2) If $C$ is a graded coalgebra, then
\begin{equation}
\Delta _{C}\circ i_{n}^{C}=\sum\limits_{a+b=n}\left( i_{a}^{C}\otimes
i_{b}^{C}\right) \circ \Delta _{a,b}^{C}  \label{form: grCoalg}
\end{equation}%
holds, $\varepsilon _{C}=\varepsilon _{C}i_{0}^{C}p_{0}^{C}$ so that $%
\varepsilon _{C}$ is a graded homomorphism, and we have that (\ref{form:
locDelta}) and (\ref{form: locEps}) hold for every $a,b,c\in
\mathbb{N}
$, where $\varepsilon _{0}^{C}=\varepsilon _{C}i_{0}$.\newline
Moreover $\left( C_{0},\Delta _{0}=\Delta _{0,0}^{C},\varepsilon
_{0}^{C}=\varepsilon _{C}i_{0}^{C}\right) $ is a coalgebra in $\mathcal{M}$,
$i_{0}^{C}$ is a coalgebra homomorphism and, for every $n\in
\mathbb{N}
$, $\left( C_{n},\Delta _{0,n}^{C},\Delta _{n,0}^{C}\right) $ is a $C_{0}$%
-bicomodule such that $p_{n}^{C}:C\rightarrow C_{n}$ is a morphism of $C_{0}$%
-bicomodules ($C$ is a $C_{0}$-bicomodule through $p_{0}^{C}$).
\end{proposition}

\begin{lemma}
\label{lem: GrDirect coalg}Let $\mathcal{M}$ be a coabelian monoidal
category such that the tensor product commutes with direct sums.

Let $((C_{a})_{a\in \mathbb{N}},(\beta _{C_{a}}^{C_{b}})_{a,b\in \mathbb{N}%
}) $ be a direct system in $\mathcal{M}$, where, for $a\leq b$, $\beta
_{C_{a}}^{C_{b}}:C_{a}\rightarrow C_{b}$ is an epimorphism.\ Assume that
there exists a family $\left( \Delta _{a,b}^{C}\right) _{_{a,b\in
\mathbb{N}
}}$
\begin{equation*}
\Delta _{a,b}^{C}:C_{a+b}\rightarrow C_{a}\otimes C_{b},
\end{equation*}%
of morphisms and a morphism $\varepsilon _{0}^{C}:C_{0}\rightarrow \mathbf{1}
$ which satisfy (\ref{form: locDelta}), (\ref{form: locEps}),
\begin{equation}
\left( \beta _{C_{a}}^{C_{a+1}}\otimes C_{b}\right) \circ \Delta
_{a,b}^{C}=\Delta _{a+1,b}^{C}\circ \beta _{C_{a+b}}^{C_{a+b+1}}\qquad \text{%
and}\qquad \left( C_{a}\otimes \beta _{C_{b}}^{C_{b+1}}\right) \circ \Delta
_{a,b}^{C}=\Delta _{a,b+1}^{C}\circ \beta _{C_{a+b}}^{C_{a+b+1}}
\label{form: locComp Coalg}
\end{equation}%
for every $a,b,c\in
\mathbb{N}
.$ Set $C_{-1}:=0.$

Let $\left( E_{n},i_{E_{n}}^{C_{n}}\right) :=\ker \left( \beta
_{C_{n}}^{C_{n+1}}\right) $ for every $n\in
\mathbb{N}
$.

Then $C=\oplus _{n\in \mathbb{N}}C_{n}$ is a graded coalgebra, there is a
unique coalgebra structure on $\oplus _{n\in \mathbb{N}}E_{n}$ such that
\begin{equation*}
\oplus _{n\in \mathbb{N}}i_{E_{n}}^{C_{n}}:\oplus _{n\in \mathbb{N}%
}E_{n}\rightarrow \oplus _{n\in \mathbb{N}}C_{n}
\end{equation*}%
is a coalgebra homomorphism and

\begin{enumerate}
\item[1)] $E=\oplus _{n\in \mathbb{N}}E_{n}$ is a graded coalgebra such that
$\oplus _{n\in \mathbb{N}}i_{E_{n}}^{C_{n}}$ is a graded homomorphism;

\item[2)] $\left( i_{E_{a}}^{C_{a}}\otimes i_{E_{b}}^{C_{b}}\right) \circ
\Delta _{a,b}^{E}=\Delta _{a,b}^{C}\circ i_{E_{a+b}}^{C_{a+b}};$

\item[3)] $\varepsilon _{E}=\varepsilon _{0}^{C}\circ i_{E_{0}}^{C_{0}}\circ
p_{0}^{E}.$
\end{enumerate}
\end{lemma}

\begin{proof}
By Proposition \ref{lem: graded Deltaij}, there exists a unique morphism $%
\Delta _{C}:C\rightarrow C\otimes C$ such that (\ref{form: grCoalg}) holds.
Moreover $\left( C=\oplus _{n\in \mathbb{N}}C_{n},\Delta _{C},\varepsilon
_{C}=\varepsilon _{0}^{C}p_{0}^{C}\right) $ is a graded coalgebra.

By left exactness of the tensor functors, we have the exact sequence%
\begin{equation}
0\rightarrow E_{a}\otimes C_{b}\overset{i_{E_{a}}^{C_{a}}\otimes C_{b}}{%
\longrightarrow }C_{a}\otimes C_{b}\overset{\beta _{C_{a}}^{C_{a+1}}\otimes
C_{b}}{\longrightarrow }C_{a+1}\otimes C_{b}.  \label{form: GrDirect1Co}
\end{equation}%
From%
\begin{equation}
\left( \beta _{C_{a}}^{C_{a+1}}\otimes C_{b}\right) \circ \Delta
_{a,b}^{C}\circ i_{E_{a+b}}^{C_{a+b}}\overset{\text{(\ref{form: locComp
Coalg})}}{=}\Delta _{a+1,b}^{C}\circ \beta _{C_{a+b}}^{C_{a+b+1}}\circ
i_{E_{a+b}}^{C_{a+b}}=0.  \label{form: GrDirect2Co}
\end{equation}%
and, by exactness of (\ref{form: GrDirect1Co}), there is a unique morphism $%
\alpha _{a,b}:E_{a+b}\rightarrow E_{a}\otimes C_{b}$ such that%
\begin{equation*}
\left( i_{E_{a}}^{C_{a}}\otimes C_{b}\right) \circ \alpha _{a,b}=\Delta
_{a,b}^{C}\circ i_{E_{a+b}}^{C_{a+b}}.
\end{equation*}

By left exactness of the tensor functors, we have the exact sequence%
\begin{equation}
0\rightarrow E_{a}\otimes E_{b}\overset{E_{a}\otimes i_{E_{b}}^{C_{b}}}{%
\longrightarrow }E_{a}\otimes C_{b}\overset{E_{a}\otimes \beta
_{C_{b}}^{C_{b+1}}}{\longrightarrow }E_{a}\otimes C_{b+1}.
\label{form: GrDirect3Co}
\end{equation}%
We obtain%
\begin{eqnarray*}
\left( i_{E_{a}}^{C_{a}}\otimes C_{b+1}\right) \circ \left( E_{a}\otimes
\beta _{C_{b}}^{C_{b+1}}\right) \circ \alpha _{a,b} &=&\left( C_{a}\otimes
\beta _{C_{b}}^{C_{b+1}}\right) \circ \left( i_{E_{a}}^{C_{a}}\otimes
C_{b}\right) \circ \alpha _{a,b} \\
&=&\left( C_{a}\otimes \beta _{C_{b}}^{C_{b+1}}\right) \circ \Delta
_{a,b}^{C}\circ i_{E_{a+b}}^{C_{a+b}}=0
\end{eqnarray*}%
where the last equality is analogue to (\ref{form: GrDirect2Co}). Since $%
i_{E_{a}}^{C_{a}}\otimes C_{b+1}$ is a monomorphism, we deduce that $\left(
E_{a}\otimes \beta _{C_{b}}^{C_{b+1}}\right) \circ \alpha _{a,b}=0$ so that,
by exactness of (\ref{form: GrDirect3Co}), there is a unique morphism $%
\Delta _{a,b}^{E}:E_{a+b}\rightarrow E_{a}\otimes E_{b},$ such that%
\begin{equation*}
\left( E_{a}\otimes i_{E_{b}}^{C_{b}}\right) \circ \Delta _{a,b}^{E}=\alpha
_{a,b}.
\end{equation*}%
We compute%
\begin{equation*}
\left( i_{E_{a}}^{C_{a}}\otimes i_{E_{b}}^{C_{b}}\right) \circ \Delta
_{a,b}^{E}=\left( i_{E_{a}}^{C_{a}}\otimes C_{b}\right) \circ \left(
E_{a}\otimes i_{E_{b}}^{C_{b}}\right) \circ \Delta _{a,b}^{E}=\left(
i_{E_{a}}^{C_{a}}\otimes C_{b}\right) \circ \alpha _{a,b}=\Delta
_{a,b}^{C}\circ i_{E_{a+b}}^{C_{a+b}}
\end{equation*}

so that $2)$ holds true.

Let us prove that $\Delta _{a,b}^{E}$ fulfills (\ref{form: locDelta}). By $%
2),$ we infer%
\begin{eqnarray*}
\left( i_{E_{a}}^{C_{a}}\otimes i_{E_{b}}^{C_{b}}\otimes
i_{E_{c}}^{C_{c}}\right) \circ \left( \Delta _{a,b}^{E}\otimes E_{c}\right)
\circ \Delta _{a+b,c}^{E} &=&\left( \Delta _{a,b}^{C}\otimes C_{c}\right)
\circ \Delta _{a+b,c}^{C}\circ i_{E_{a+b+c}}^{C_{a+b+c}}, \\
\left( i_{E_{a}}^{C_{a}}\otimes i_{E_{b}}^{C_{b}}\otimes
i_{E_{c}}^{C_{c}}\right) \circ \left( E_{a}\otimes \Delta _{b,c}^{E}\right)
\circ \Delta _{a,b+c}^{E} &=&\left( C_{a}\otimes \Delta _{b,c}^{C}\right)
\circ \Delta _{a,b+c}^{C}\circ i_{E_{a+b+c}}^{C_{a+b+c}}.
\end{eqnarray*}%
Since $C$ fulfills (\ref{form: locDelta}) and since $i_{E_{a}}^{C_{a}}%
\otimes i_{E_{b}}^{C_{b}}\otimes i_{E_{c}}^{C_{c}}$ is a monomorphism, we
get (\ref{form: locDelta}) for $E$. \newline

Let $\varepsilon _{0}^{E}:=\varepsilon _{0}^{C}\circ i_{E_{0}}^{C_{0}}.$ Let
us prove that (\ref{form: locEps}) hold for $E$. By $2),$ we have%
\begin{eqnarray*}
\left( i_{E_{d}}^{C_{d}}\otimes \mathbf{1}\right) \circ \left( E_{d}\otimes
\varepsilon _{0}^{E}\right) \circ \Delta _{d,0}^{E} &=&\left( C_{d}\otimes
\varepsilon _{0}^{C}\right) \circ \left( i_{E_{d}}^{C_{d}}\otimes
i_{E_{0}}^{C_{0}}\right) \circ \Delta _{d,0}^{E}=\left( C_{d}\otimes
\varepsilon _{0}^{C}\right) \circ \Delta _{d,0}^{C}\circ i_{E_{d}}^{C_{d}},
\\
\left( i_{E_{d}}^{C_{d}}\otimes \mathbf{1}\right) \circ r_{E_{d}}^{-1}
&=&r_{C_{d}}^{-1}\circ i_{E_{d}}^{C_{d}}.
\end{eqnarray*}%
By (\ref{form: locEps}) for $C$ and since $i_{E_{d}}^{C_{d}}\otimes \mathbf{1%
}$ is a monomorphism, we get the left equation of (\ref{form: locEps}) for $%
E $. \newline
Similarly one gets the other equation. Thus, by applying Proposition \ref%
{lem: graded Deltaij}, we conclude that $E$ is a graded coalgebra and $3)$
holds true.

It remains to prove that $i:=\oplus _{n\in \mathbb{N}}i_{E_{n}}^{C_{n}}$ is
a coalgebra homomorphism. For every $a,b\in
\mathbb{N}
$, we have%
\begin{eqnarray*}
&&\left( i\otimes i\right) \circ \Delta _{E}\circ i_{n}^{E}\overset{(\ref%
{form: grCoalg})}{=}\left( i\otimes i\right) \circ \sum_{a+b=n}\left(
i_{a}^{E}\otimes i_{b}^{E}\right) \circ \Delta _{a,b}^{E} \\
&=&\sum_{a+b=n}\left( i_{a}^{C}\otimes i_{b}^{C}\right) \circ \left(
i_{E_{a}}^{C_{a}}\otimes i_{E_{b}}^{C_{b}}\right) \circ \Delta _{a,b}^{E}%
\overset{2)}{=}\sum_{a+b=n}\left( i_{a}^{C}\otimes i_{b}^{C}\right) \circ
\Delta _{a,b}^{C}\circ i_{E_{a+b}}^{C_{a+b}} \\
&&\overset{(\ref{form: grCoalg})}{=}\Delta _{C}\circ i_{n}^{C}\circ
i_{E_{n}}^{C_{n}}=\Delta _{C}\circ i\circ i_{n}^{E}
\end{eqnarray*}%
and%
\begin{equation*}
\varepsilon _{C}\circ i\circ i_{n}^{E}=\varepsilon _{C}\circ i_{n}^{C}\circ
i_{E_{n}}^{C_{n}}=\varepsilon _{0}^{C}\circ p_{0}^{C}\circ i_{n}^{C}\circ
i_{E_{n}}^{C_{n}}=\delta _{n,0}\varepsilon _{0}^{C}\circ
i_{E_{0}}^{C_{0}}=\varepsilon _{0}^{C}\circ i_{E_{0}}^{C_{0}}\circ
p_{0}^{E}\circ i_{n}^{E}=\varepsilon _{E}\circ i_{n}^{E}
\end{equation*}%
so that $\left( i\otimes i\right) \circ \Delta _{E}=\Delta _{C}\circ i$ and $%
\varepsilon _{C}\circ i=\varepsilon _{E}.$ Thus $i$ is a coalgebra
homomorphism.
\end{proof}

\begin{lemma}
\label{lem: GrDirect coalg Inverse}Let $\mathcal{M}$ be a coabelian monoidal
category such that the tensor product commutes with direct sums.

Let $((C_{a})_{a\in \mathbb{N}},(\beta _{C_{a}}^{C_{b}})_{a,b\in \mathbb{N}%
}) $ be an inverse system in $\mathcal{M}$, where, for $a\leq b$, $\beta
_{C_{b}}^{C_{a}}:C_{b}\rightarrow C_{a}$ is an epimorphism.\ Assume that
there exists a family $\left( \Delta _{a,b}^{C}\right) _{_{a,b\in
\mathbb{N}
}}$
\begin{equation*}
\Delta _{a,b}^{C}:C_{a+b}\rightarrow C_{a}\otimes C_{b},
\end{equation*}%
of morphisms and a morphism $\varepsilon _{0}^{C}:C_{0}\rightarrow \mathbf{1}
$ which satisfy (\ref{form: locDelta}), (\ref{form: locEps}),
\begin{equation}
\left( \beta _{C_{a+1}}^{C_{a}}\otimes C_{b}\right) \circ \Delta
_{a+1,b}^{C}=\Delta _{a,b}^{C}\circ \beta _{C_{a+b+1}}^{C_{a+b}}\qquad \text{%
and}\qquad \left( C_{a}\otimes \beta _{C_{b+1}}^{C_{b}}\right) \circ \Delta
_{a,b+1}^{C}=\Delta _{a,b}^{C}\circ \beta _{C_{a+b+1}}^{C_{a+b}}
\label{form: locComp Coalg Inv}
\end{equation}%
for every $a,b,c\in
\mathbb{N}
.$ Set $C_{-1}:=0.$

Let $\left( E_{n},i_{E_{n}}^{C_{n}}\right) :=\ker \left( \beta
_{C_{n}}^{C_{n-1}}\right) $ for every $n\in
\mathbb{N}
$.

Then $C=\oplus _{n\in \mathbb{N}}C_{n}$ is a graded coalgebra, there is a
unique coalgebra structure on $\oplus _{n\in \mathbb{N}}E_{n}$ such that
\begin{equation*}
\oplus _{n\in \mathbb{N}}i_{E_{n}}^{C_{n}}:\oplus _{n\in \mathbb{N}%
}E_{n}\rightarrow \oplus _{n\in \mathbb{N}}C_{n}
\end{equation*}%
is a coalgebra homomorphism and

\begin{enumerate}
\item[1)] $E=\oplus _{n\in \mathbb{N}}E_{n}$ is a graded coalgebra such that
$\oplus _{n\in \mathbb{N}}i_{E_{n}}^{C_{n}}$ is a graded homomorphism;

\item[2)] $\left( i_{E_{a}}^{C_{a}}\otimes i_{E_{b}}^{C_{b}}\right) \circ
\Delta _{a,b}^{E}=\Delta _{a,b}^{C}\circ i_{E_{a+b}}^{C_{a+b}};$

\item[3)] $\varepsilon _{E}=\varepsilon _{0}^{C}\circ i_{E_{0}}^{C_{0}}\circ
p_{0}^{E}.$
\end{enumerate}
\end{lemma}

\begin{proof}
It is similar to that of Lemma \ref{lem: GrDirect coalg}.
\end{proof}

\begin{lemma}
\label{lem: Bourbaki co}Consider the following commutative diagram in an
abelian category $\mathcal{C}$.
\begin{equation*}
\xymatrix{ 0 \ar[r] & \mathrm{ker}(\beta) \ar[r]^{i} & Z \ar[d]_{g}
\ar[r]^{\beta} & Y \ar[d]^{f} \\ & & W \ar[r]^{\alpha} & X }
\end{equation*}%
Assume that both $f$ and $g\circ i$ are monomorphisms. Then $g$ is a
monomorphism.
\end{lemma}

\begin{proof}
Let $\xi :T\rightarrow Z$ be a morphism such that $g\circ \xi =0.$ Then $%
f\circ \beta \circ \xi =\alpha \circ g\circ \xi =0$ so that, since $f$ is a
monomorphism, we get $\beta \circ \xi =0.$ By the universal property of
kernels $\xi $ factors to a map $\overline{\xi }:T\rightarrow \mathrm{\ker }%
\left( \beta \right) $ such that $i\circ \overline{\xi }=\xi .$

Now $g\circ i\circ \overline{\xi }=g\circ \xi =0$ so that, since $g\circ i$
is a monomorphism, we get $\overline{\xi }=0.$
\end{proof}

\begin{theorem}
\label{teo: strongly Alg co}With hypothesis and notations of Lemma \ref{lem:
GrDirect coalg Inverse}, the following assertions are equivalent.

\begin{enumerate}
\item[$\left( 1\right) $] $C=\oplus _{n\in \mathbb{N}}C_{n}$ is a strongly $%
\mathbb{N}
$-graded coalgebra.

\item[$\left( 2\right) $] $E=\oplus _{n\in \mathbb{N}}E_{n}$ is a strongly $%
\mathbb{N}
$-graded coalgebra.
\end{enumerate}
\end{theorem}

\begin{proof}
Let $((C_{a})_{a\in \mathbb{N}},(\beta _{C_{a}}^{C_{b}})_{a,b\in \mathbb{N}%
}) $ be an inverse system in $\mathcal{M}$, where, for $a\leq b$, $\beta
_{C_{b}}^{C_{a}}:C_{b}\rightarrow C_{a}$ is an epimorphism.\ Assume that
there exists a family $\left( \Delta _{a,b}^{C}\right) _{_{a,b\in
\mathbb{N}
}}$
\begin{equation*}
\Delta _{a,b}^{C}:C_{a+b}\rightarrow C_{a}\otimes C_{b},
\end{equation*}%
of morphisms and a morphism $\varepsilon _{0}^{C}:C_{0}\rightarrow \mathbf{1}
$ which satisfy (\ref{form: locDelta}), (\ref{form: locEps}) and (\ref{form:
locComp Coalg Inv}). Let $\left( E_{n},i_{E_{n}}^{C_{n}}\right) :=\ker
\left( \beta _{C_{n}}^{C_{n-1}}\right) $ for every $n\in
\mathbb{N}
$.

$\left( 1\right) \Rightarrow \left( 2\right) $ It follows from $\left(
i_{E_{a}}^{C_{a}}\otimes i_{E_{b}}^{C_{b}}\right) \circ \Delta
_{a,b}^{E}=\Delta _{a,b}^{C}\circ i_{E_{a+b}}^{C_{a+b}}$ which holds in view
of Lemma \ref{lem: GrDirect coalg Inverse}.

$\left( 2\right) \Rightarrow \left( 1\right) $ By assumption, $\Delta
_{a,b}^{E}$ is a monomorphism for every $a,b\in
\mathbb{N}
.$

In view of \cite[Theorem 2.22]{AM}, it is enough to prove that $\Delta
_{a,1}^{C}:C_{a+1}\rightarrow C_{a}\otimes C_{1}$ is a monomorphism by
induction on $a\in
\mathbb{N}
$.

$a=0)$ From%
\begin{equation*}
\left( \varepsilon _{0}\otimes C_{1}\right) \circ \Delta _{0,1}^{C}\overset{%
\text{(\ref{form: locEps})}}{=}l_{C_{1}}^{-1}
\end{equation*}%
we deduce that $\Delta _{0,1}^{C}$ is a monomorphism.

$a-1\Rightarrow a)$ Since $\Delta _{a,1}^{E}$ and $i_{E_{a}}^{C_{a}}\otimes
i_{E_{1}}^{C_{1}}$ are monomorphisms and%
\begin{equation*}
\left( i_{E_{a}}^{C_{a}}\otimes i_{E_{1}}^{C_{1}}\right) \circ \Delta
_{a,1}^{E}=\Delta _{a,1}^{C}\circ i_{E_{a+1}}^{C_{a+1}}
\end{equation*}
which holds in view of Lemma \ref{lem: GrDirect coalg Inverse}, we deduce
that $\Delta _{a,1}^{C}\circ i_{E_{a+1}}^{C_{a+1}}$ is a monomorphism too.
From (\ref{form: locComp Coalg Inv}), we get%
\begin{equation*}
\left( \beta _{C_{a}}^{C_{a-1}}\otimes C_{1}\right) \circ \Delta
_{a,1}^{C}=\Delta _{a-1,1}^{C}\circ \beta _{C_{a+1}}^{C_{a}}.
\end{equation*}%
Hence by inductive hypothesis, we can apply Lemma \ref{lem: Bourbaki co} to
the following commutative diagram:
\begin{equation*}
\xymatrix{ 0 \ar[r] & E_{a+1} \ar[r]^{i_{E_{a+1}}^{C_{a+1}}} & C_{a+1}
\ar[d]_{\Delta _{a,1}^{C}} \ar[rr]^{\beta _{C_{a+1}}^{C_{a}}} && C_{a}
\ar[d]^{\Delta _{a-1,1}^{C}} \\ & & C_a\otimes C_1 \ar[rr]^{\beta
_{C_{a}}^{C_{a-1}}\otimes C_{1}} && C_{a-1}\otimes C_1 }
\end{equation*}
\end{proof}

\begin{claim}
\label{claim: wedge}Let $\mathcal{M}$ be a coabelian monoidal category.%
\newline
Let $(C,i_{C}^{E})$ and $(D,i_{D}^{E})$ be two subobjects of a coalgebra $%
(E,\Delta ,\varepsilon )$. Set
\begin{gather*}
\Delta _{C,D}:=(p_{C}^{E}\otimes p_{D}^{E})\Delta :E\rightarrow \frac{E}{C}%
\otimes \frac{E}{D} \\
(C\wedge _{E}D,i_{C\wedge _{E}D}^{E})=\ker \left( \Delta _{C,D}\right)
,\qquad i_{C\wedge _{E}D}^{E}:C\wedge _{E}D\rightarrow E \\
(\frac{E}{C\wedge _{E}D},p_{C\wedge _{E}D}^{E})=\mathrm{\mathrm{Coker}}%
\left( i_{C\wedge _{E}D}^{E}\right) =\text{\textrm{Im}}\left( \Delta
_{C,D}\right) ,\qquad p_{C\wedge _{E}D}^{E}:E\rightarrow \frac{E}{C\wedge
_{E}D}
\end{gather*}%
Moreover, we have the following exact sequence:%
\begin{equation}
0\longrightarrow C\wedge _{E}D\overset{i_{C\wedge _{E}D}^{E}}{%
\longrightarrow }E\overset{p_{C\wedge _{E}D}^{E}}{\longrightarrow }\frac{E}{%
C\wedge _{E}D}\longrightarrow 0.  \label{ec:wedge}
\end{equation}%
Since $(\frac{E}{C\wedge _{E}D},p_{C\wedge _{E}D}^{E})=\mathrm{\mathrm{Coker}%
}\left( i_{C\wedge _{E}D}^{E}\right) \ $and $\Delta _{C,D}\circ i_{C\wedge
_{E}D}^{E}=0$, by the universal property of the cokernel, there is a unique
morphism $\overline{\Delta }_{C,D}:\frac{E}{C\wedge _{E}D}\rightarrow \frac{E%
}{C}\otimes \frac{E}{D}$ such that the following diagram
\begin{equation*}
\diagComultWedge%
\end{equation*}%
is commutative. Since $(\frac{E}{C\wedge _{E}D},p_{C\wedge _{E}D}^{E})=$%
\textrm{Im}$\left( \Delta _{C,D}\right) ,$ it comes out that $\overline{%
\Delta }_{C,D}$ is a monomorphism.\newline
Assume now that $(C,i_{C}^{E})$ and $(D,i_{D}^{E})$ are two subcoalgebras of
$(E,\Delta ,\varepsilon )$. Since $\Delta _{C,D}\in {^{E}{\mathcal{M}}^{E}}{,%
}$ it is straightforward to prove that $C\wedge _{E}D$ is a coalgebra and
that $i_{C\wedge _{E}D}^{E}$ is a coalgebra homomorphism.\newline
Consider the case $C=0.$\newline
Since $p_{D}^{E}$ is a morphism in ${^{E}{\mathcal{M}}}$, we have
\begin{equation*}
\Delta _{0,D}=(\mathrm{Id}_{E}\otimes p_{D}^{E})\circ \Delta =\rho
_{E/D}^{l}\circ p_{D}^{E}.
\end{equation*}%
Since $\rho _{E/D}^{l}$ is a monomorphism, we deduce that%
\begin{equation*}
(0\wedge _{E}D,i_{0\wedge _{E}D}^{E})=\ker \left( \Delta _{0,D}\right) =\ker
\left( p_{D}^{E}\right) =\left( D,i_{D}^{E}\right) .
\end{equation*}%
Analogously, in the case $D=0,$ one has%
\begin{equation*}
(C\wedge _{E}0,i_{C\wedge _{E}0}^{E})=\left( C,i_{C}^{E}\right) .
\end{equation*}
\end{claim}

\begin{claim}
Let $(C,i_{C}^{E})$ be a subobject of a coalgebra $(E,\Delta ,\varepsilon )$
in a coabelian monoidal category $\mathcal{M}$. We can define (see \cite{AMS}%
) the $n$-th wedge product $\left( C^{\wedge _{E}n},i_{C^{\wedge
_{E}n}}^{E}\right) $ of $C$ in $E$ where $i_{C^{\wedge _{E}n}}^{E}:C^{\wedge
_{E}n}\rightarrow E.$ By definition, we have
\begin{equation*}
C^{\wedge _{E}0}=0\qquad \text{and}\qquad C^{\wedge _{E}n}=C^{\wedge
_{E}n-1}\wedge _{E}C,\text{ for every }n\geq 1.
\end{equation*}%
One can check that $\left( \left( C\wedge
_{E}D\right) \wedge _{E}F,i_{\left( C\wedge _{E}D\right) \wedge
_{E}F}^{E}\right) $ and $\left( C\wedge _{E}\left( D\wedge _{E}F\right)
,i_{C\wedge _{E}\left( D\wedge _{E}F\right) }^{E}\right) $ are isomorphic, for every subobject $C,D,F$ of $E$, 
and thus can be identified. 
Therefore $C^{\wedge _{E}i}\wedge _{E}C^{\wedge _{E}j}=C^{\wedge _{E}i+j}$
and we can consider
\begin{equation*}
\overline{\Delta }_{C^{\wedge _{E}i},C^{\wedge _{E}j}}:\frac{E}{C^{\wedge
_{E}i+j}}\rightarrow \frac{E}{C^{\wedge _{E}i}}\otimes \frac{E}{C^{\wedge
_{E}j}}.
\end{equation*}%
Assume now that $(C,i_{C}^{E})$ is a subcoalgebra of the coalgebra $%
(E,\Delta ,\varepsilon ).$ Then there is a (unique) coalgebra homomorphism%
\begin{equation*}
i_{C^{\wedge _{E}n}}^{C^{\wedge _{E}n+1}}:C^{\wedge _{E}n}\rightarrow
C^{\wedge _{E}n+1},\text{ for every }n\in
\mathbb{N}
.
\end{equation*}%
such that $i_{C^{\wedge _{E}n+1}}^{E}\circ i_{C^{\wedge _{E}n}}^{C^{\wedge
_{E}n+1}}=i_{C^{\wedge _{E}n}}^{E}.$ We set%
\begin{equation*}
E_{C}:=\oplus _{n\in \mathbb{N}}\frac{E}{C^{\wedge _{E}n}}
\end{equation*}%
and
\begin{equation*}
\Delta _{i,j}^{E_{C}}:=\overline{\Delta }_{C^{\wedge _{E}i},C^{\wedge
_{E}j}}:\frac{E}{C^{\wedge _{E}i+j}}\rightarrow \frac{E}{C^{\wedge _{E}i}}%
\otimes \frac{E}{C^{\wedge _{E}j}}.
\end{equation*}%
Since $\Delta $ is coassociative and by definition of $\Delta _{i,j}^{E_{C}}$%
, it is straightforward to prove that $\Delta _{i,j}^{E_{C}}$ fulfills
\begin{gather}
\left( \Delta _{a,b}^{E_{C}}\otimes \frac{E}{C^{\wedge _{E}c}}\right) \Delta
_{a+b,c}^{E_{C}}=\left( \frac{E}{C^{\wedge _{E}a}}\otimes \Delta
_{b,c}^{E_{C}}\right) \Delta _{a,b+c}^{E_{C}}.  \label{form: locComultiWedge}
\\
\left( \frac{E}{C^{\wedge _{E}d}}\otimes \varepsilon \right) \Delta
_{d,0}^{E_{C}}=r_{\frac{E}{C^{\wedge _{E}d}}}^{-1},\qquad \left( \varepsilon
\otimes \frac{E}{C^{\wedge _{E}d}}\right) \Delta _{0,d}^{E_{C}}=l_{\frac{E}{%
C^{\wedge _{E}d}}}^{-1}  \label{form: locCounitWedge} \\
\left( \frac{E}{i_{C^{\wedge _{E}a}}^{C^{\wedge _{E}a+1}}}\otimes
C_{b}\right) \circ \Delta _{a,b}^{E_{C}}=\Delta _{a+1,b}^{E_{C}}\circ \frac{E%
}{i_{C^{\wedge _{E}a+b}}^{C^{\wedge _{E}a+b+1}}},\qquad \left( C_{a}\otimes
\frac{E}{i_{C^{\wedge _{E}b}}^{C^{\wedge _{E}b+1}}}\right) \circ \Delta
_{a,b}^{E_{C}}=\Delta _{a,b+1}^{E_{C}}\circ \frac{E}{i_{C^{\wedge
_{E}a+b}}^{C^{\wedge _{E}a+b+1}}}.  \label{form: locComp CoalgWedge}
\end{gather}
\end{claim}

\begin{theorem}
\label{teo: gr dual}Let $\mathcal{M}$ be a cocomplete coabelian monoidal
category such that the tensor product commutes with direct sums.

Let $(C,i_{C}^{E})$ be a subcoalgebra of a coalgebra $(E,\Delta ,\varepsilon
)$ in a coabelian monoidal category $\mathcal{M}$. For every $n\in
\mathbb{N}
$, we set
\begin{equation*}
gr_{C}^{n}E=\frac{C^{\wedge _{E}n+1}}{C^{\wedge _{E}n}}.
\end{equation*}%
Then $E_{C}=\oplus _{n\in \mathbb{N}}\frac{E}{C^{\wedge _{E}n}}$ is a graded
coalgebra, there is a unique coalgebra structure on $gr_{C}E:=\oplus _{n\in
\mathbb{N}}gr_{C}^{n}E$ such that
\begin{equation*}
\oplus _{n\in \mathbb{N}}\frac{i_{C^{\wedge _{E}n+1}}^{E}}{C^{\wedge _{E}n}}%
:gr_{C}E\rightarrow E_{C}
\end{equation*}%
is a coalgebra homomorphism and

\begin{enumerate}
\item $gr_{C}E$ is a graded coalgebra such that $\oplus _{n\in \mathbb{N}}%
\frac{i_{C^{\wedge _{E}n+1}}^{E}}{C^{\wedge _{E}n}}$ is a graded
homomorphism;

\item
\begin{equation}
\left( \frac{i_{C^{\wedge _{E}a+1}}^{E}}{C^{\wedge _{E}a}}\otimes \frac{%
i_{C^{\wedge _{E}b+1}}^{E}}{C^{\wedge _{E}b}}\right) \circ \Delta
_{a,b}^{gr_{C}E}=\Delta _{a,b}^{E_{C}}\circ \frac{i_{C^{\wedge
_{E}a+b+1}}^{E}}{C^{\wedge _{E}a+b}};  \label{form: Delta grCo}
\end{equation}

\item $\varepsilon _{gr_{C}E}=\varepsilon _{E}\circ i_{C}^{E}\circ
p_{0}^{gr_{C}E}.$
\end{enumerate}

Moreover $gr_{C}E$ is a strongly $%
\mathbb{N}
$-graded coalgebra.
\end{theorem}

\begin{proof}
By (\ref{form: locComultiWedge}), (\ref{form: locCounitWedge}) and (\ref%
{form: locComp CoalgWedge}), we can apply Lemma \ref{lem: GrDirect coalg} to
the family $(\frac{E}{C^{\wedge _{E}n}})_{n\in
\mathbb{N}
}.$ It remains to prove the last assertion.

From 2), since both $\frac{i_{C^{\wedge _{E}a+b+1}}^{E}}{C^{\wedge _{E}a+b}}$
and $\Delta _{a,b}^{E_{C}}$ are monomorphisms, we get that $\Delta
_{a,b}^{gr_{C}E}$ is a monomorphism too, for every $a,b\in \mathbb{N}.$ Thus
$gr_{C}E$ is a strongly $%
\mathbb{N}
$-graded coalgebra.
\end{proof}

\begin{definition}
Let $(C,i_{C}^{E})$ be a subcoalgebra of a coalgebra $(E,\Delta ,\varepsilon
)$ in a cocomplete coabelian monoidal category $\mathcal{M}$ such that the
tensor product commutes with direct sums.

The strongly $%
\mathbb{N}
$-graded coalgebra $gr_{C}E$ defined in Theorem \ref{teo: gr dual} will be
called the \emph{associated graded coalgebra }(of $E$ with respect to $C$).
\end{definition}

\begin{theorem}
\label{teo: graded dual}Let $\mathcal{M}$ be a cocomplete and complete
coabelian monoidal category satisfying $AB5$ such that the tensor product
commutes with direct sums. Let $(C,i_{C}^{E})$ be a subcoalgebra of a
coalgebra $(E,\Delta ,\varepsilon )$ in $\mathcal{M}$ and let $gr_{C}E$ be
the associated graded coalgebra.

Let
\begin{equation*}
T^{c}:=T_{C}^{c}\left( \frac{C\wedge _{E}C}{C}\right)
\end{equation*}%
be the cotensor coalgebra. Then there is a unique coalgebra homomorphism
\begin{equation*}
\psi :gr_{C}E\rightarrow T_{C}^{c}\left( \frac{C\wedge _{E}C}{C}\right) ,
\end{equation*}%
such that $p_{0}^{T^{c}}\circ \psi =p_{0}^{gr_{C}E}$ and $p_{1}^{T^{c}}\circ
\psi =p_{1}^{gr_{C}E}$. \newline
Moreover $\psi $ is a graded coalgebra homomorphism with%
\begin{equation*}
\psi _{m}=\left( p_{1}^{gr_{C}E}\right) ^{\square m}\circ \overline{\Delta }%
_{gr_{C}E}^{m-1}\circ i_{m}^{gr_{C}E}\text{ for every }m\in
\mathbb{N}%
\end{equation*}%
and the following equivalent assertions hold.

$\left( a\right) $ $gr_{C}E$ is a strongly $%
\mathbb{N}
$-graded coalgebra.

$\left( a^{\prime }\right) $ $\Delta _{a,1}^{gr_{C}E}:gr_{C}^{a+1}\left(
E\right) \rightarrow gr_{C}^{a}\left( E\right) \otimes gr_{C}^{1}\left(
E\right) $ is a monomorphism for every $a\in
\mathbb{N}
$.

$\left( b\right) $ $\psi _{n}$ is a monomorphism for every $n\in
\mathbb{N}
$.

$\left( c\right) $ $\psi $ is a monomorphism.

$\left( d\right) $ $\oplus _{0\leq i\leq n-1}gr_{C}^{i}E=C^{\wedge
_{gr_{C}E}^{n}},$ for every $n\geq 1$.

$\left( e\right) $ $\oplus _{0\leq i\leq 1}gr_{C}^{i}E=C\oplus \frac{C\wedge
_{E}C}{C}=C^{\wedge _{gr_{C}E}^{2}}$.
\end{theorem}

\begin{proof}
By Theorem \ref{teo: gr dual}, $\left( a\right) $ holds. We conclude by
applying \cite[Theorem 2.22]{AM}.
\end{proof}

\section{The Associated Graded Algebra\label{sec: ass Grad alg}}

\begin{claim}
\label{def: grAlg} Let $\mathcal{M}$ be a cocomplete abelian monoidal
category such that the tensor product commutes with direct sums.

Recall that a \emph{graded algebra} in $\mathcal{M}$ is an algebra $\left(
A,m,u\right) $ where
\begin{equation*}
A=\oplus _{n\in \mathbb{N}}A_{n}
\end{equation*}%
is a graded object of $\mathcal{M}$ such that $m:A\otimes A\rightarrow A$ is
a graded homomorphism i.e. there exists a family $\left( m_{n}\right)
_{_{n\in \mathbb{N}}}$ of morphisms
\begin{equation*}
m_{n}^{A}=m_{n}:\oplus _{a+b=n}\left( A_{a}\otimes A_{b}\right) =\left(
A\otimes A\right) _{n}\rightarrow A_{n}\text{ such that }m=\oplus _{n\in
\mathbb{N}}m_{n}.
\end{equation*}%
We set
\begin{equation*}
m_{a,b}^{A}:=\left( A_{a}\otimes A_{b}\overset{\gamma _{a,b}^{A,A}}{%
\rightarrow }\left( A\otimes A\right) _{a+b}\overset{m_{a+b}}{\rightarrow }%
A_{a+b}\right) .
\end{equation*}%
A homomorphism $f:\left( A,m_{A},u_{A}\right) \rightarrow \left(
B,m_{B},u_{B}\right) $ of algebras is a graded algebra homomorphism if it is
a graded homomorphism too.
\end{claim}

\begin{definition}
\label{def: strongly grAlg}Let $(A=\oplus _{n\in
\mathbb{N}
}A_{n},m,u)$ be a graded algebra in $\mathcal{M}$. In analogy with the group
graded case, we say that $A$ is a \emph{strongly }$%
\mathbb{N}
$\emph{-graded algebra} whenever

$m_{i,j}^{A}:A_{i}\otimes A_{j}\rightarrow A_{i+j}$ is an epimorphism for
every $i,j\in \mathbb{N},$

where $m_{i,j}^{A}$ is the morphism of Definition \ref{def: grAlg}.
\end{definition}

\begin{proposition}
\cite[Proposition 3.4]{AM} \label{lem: graded mij} Let $\mathcal{M}$ be a
cocomplete abelian monoidal category such that the tensor product commutes
with direct sums.

1) Let $A=\oplus _{n\in \mathbb{N}}A_{n}$ be a graded object of $\mathcal{M}$
such that there exists a family $\left( m_{a,b}\right) _{_{a,b\in
\mathbb{N}
}}$
\begin{equation*}
m_{a,b}^{A}:A_{a}\otimes A_{b}\rightarrow A_{a+b},
\end{equation*}%
of morphisms and a morphism $u_{0}^{A}:\mathbf{1}\rightarrow A_{0}$ which
satisfy
\begin{gather}
m_{a+b,c}^{A}\circ \left( m_{a,b}^{A}\otimes A_{c}\right)
=m_{a,b+c}^{A}\circ \left( A_{a}\otimes m_{b,c}^{A}\right) \text{,}
\label{form: locMulti} \\
m_{d,0}^{A}\circ \left( A_{d}\otimes u_{0}^{A}\right) =r_{A_{d}},\qquad
m_{0,d}^{A}\circ \left( u_{0}^{A}\otimes A_{d}\right) =l_{A_{d}}\text{,}
\label{form: locUnit}
\end{gather}%
for every $a,b,c\in
\mathbb{N}
$. Then there exists a unique morphism $m_{A}:A\otimes A\rightarrow A$ such
that
\begin{equation}
m_{A}\circ (i_{a}^{A}\otimes i_{b}^{A})=i_{a+b}^{A}\circ m_{a,b}^{A},\text{
for every }a,b\in
\mathbb{N}
\label{form: coro grAlg1}
\end{equation}%
holds.\newline
Moreover $\left( A=\oplus _{n\in \mathbb{N}}A_{n},m_{A},u_{A}=i_{0}^{A}\circ
u_{0}^{A}\right) $ is a graded algebra.

$2)$ If $A$ is a graded algebra then
\begin{equation}
p_{a+b}^{A}\circ m_{A}=\sum\limits_{a+b=n}m_{a,b}^{A}\circ \left(
p_{a}^{A}\otimes p_{b}^{A}\right)  \label{form: grAlg}
\end{equation}%
holds, $u_{A}=i_{0}^{A}p_{0}^{A}u_{A}\ $so that $u_{A}$ is a graded
homomorphism, and we have that (\ref{form: locMulti}) and (\ref{form:
locUnit}) hold for every $a,b,c\in
\mathbb{N}
$, where $u_{0}^{A}=p_{0}^{A}u_{A}$.

Moreover $\left( A_{0},m_{0}=m_{0,0}^{A},u_{0}^{A}=p_{0}^{A}u_{A}\right) $
is an algebra in $\mathcal{M}$, $p_{0}^{A}$ is an algebra homomorphism and,
for every $n\in
\mathbb{N}
$, $\left( A_{n},m_{0,n}^{A},m_{n,0}^{A}\right) $ is an $A_{0}$-bimodule
such that $i_{n}^{A}:A_{n}\rightarrow A$ is a morphism of $A_{0}$-bimodules (%
$A$ is an $A_{0}$-bimodule through $i_{0}^{A}$).
\end{proposition}

\begin{lemma}
\label{lem: GrDirect}Let $\mathcal{M}$ be a cocomplete abelian monoidal
category such that the tensor product commutes with direct sums. Let $%
((A_{a})_{a\in \mathbb{N}},(a_{A_{a}}^{A_{b}})_{a,b\in \mathbb{N}})$ be a
direct system in $\mathcal{M}$ where, for $a\leq b$, $%
i_{A_{a}}^{A_{b}}:A_{a}\rightarrow A_{b}$ is a monomorphism.\ Assume that
there exists a family $\left( m_{a,b}^{A}\right) _{_{a,b\in
\mathbb{N}
}}$
\begin{equation*}
m_{a,b}^{A}:A_{a}\otimes A_{b}\rightarrow A_{a+b},
\end{equation*}%
of morphisms and a morphism $u_{0}^{A}:\mathbf{1}\rightarrow A_{0}$ which
satisfy (\ref{form: locMulti}), (\ref{form: locUnit}),
\begin{equation}
m_{a+1,b}^{A}\circ \left( i_{A_{a}}^{A_{a+1}}\otimes A_{b}\right)
=i_{A_{a+b}}^{A_{a+b+1}}\circ m_{a,b}^{A}\qquad \text{and}\qquad
m_{a,b+1}^{A}\circ \left( A_{a}\otimes i_{A_{b}}^{A_{b+1}}\right)
=i_{A_{a+b}}^{A_{a+b+1}}\circ m_{a,b}^{A}  \label{form: locComp alg}
\end{equation}%
for every $a,b,c\in
\mathbb{N}
.$ Set $A_{-1}:=0.$

Then $A=\oplus _{n\in \mathbb{N}}A_{n}$ is a graded algebra and there are
unique algebra structure on $\oplus _{n\in \mathbb{N}}\frac{A_{n}}{A_{n-1}}$
such that
\begin{equation*}
\oplus _{n\in \mathbb{N}}p_{A_{n-1}}^{A_{n}}:\oplus _{n\in \mathbb{N}%
}A_{n}\rightarrow \oplus _{n\in \mathbb{N}}\frac{A_{n}}{A_{n-1}}
\end{equation*}%
is an algebra homomorphism. Moreover

\begin{enumerate}
\item[1)] $E=\oplus _{n\in \mathbb{N}}\frac{A_{n}}{A_{n-1}}$ is a graded
algebra such that $\oplus _{n\in \mathbb{N}}p_{A_{n-1}}^{A_{n}}$ is a graded
homomorphism;

\item[2)] $m_{a,b}^{E}\circ \left( p_{A_{a-1}}^{A_{a}}\otimes
p_{A_{b-1}}^{A_{b}}\right) =p_{A_{a+b-1}}^{A_{a+b}}\circ m_{a,b}^{A};$

\item[3)] $u_{E}=i_{0}^{E}\circ p_{A_{-1}}^{A_{0}}\circ u_{0}^{A}.$
\end{enumerate}
\end{lemma}

\begin{proof}
It is analogous to that of Lemma \ref{lem: GrDirect coalg}.
\end{proof}

\begin{lemma}
\label{lem: GrDirect Inverse}Let $\mathcal{M}$ be a cocomplete abelian
monoidal category such that the tensor product commutes with direct sums.
Let $((A_{a})_{a\in \mathbb{N}},(i_{A_{a}}^{A_{b}})_{a,b\in \mathbb{N}})$ be
an inverse system in $\mathcal{M}$ where, for $a\leq b$, $%
i_{A_{b}}^{A_{a}}:A_{b}\rightarrow A_{a}$ is a monomorphism.\ Assume that
there exists a family $\left( m_{a,b}^{A}\right) _{_{a,b\in
\mathbb{N}
}}$
\begin{equation*}
m_{a,b}^{A}:A_{a}\otimes A_{b}\rightarrow A_{a+b},
\end{equation*}%
of morphisms and a morphism $u_{0}^{A}:\mathbf{1}\rightarrow A_{0}$ which
satisfy (\ref{form: locMulti}), (\ref{form: locUnit}),
\begin{equation}
m_{a,b}^{A}\circ \left( i_{A_{a+1}}^{A_{a}}\otimes A_{b}\right)
=i_{A_{a+b+1}}^{A_{a+b}}\circ m_{a+1,b}^{A}\qquad \text{and}\qquad
m_{a,b}^{A}\circ \left( A_{a}\otimes i_{A_{b+1}}^{A_{b}}\right)
=i_{A_{a+b+1}}^{A_{a+b}}\circ m_{a,b+1}^{A}
\label{form: locComp alg Inverse}
\end{equation}%
for every $a,b,c\in
\mathbb{N}
.$ Set $A_{-1}:=0.$

Then $A=\oplus _{n\in \mathbb{N}}A_{n}$ is a graded algebra and there are
unique algebra structure on $\oplus _{n\in \mathbb{N}}\frac{A_{n}}{A_{n+1}}$
such that
\begin{equation*}
\oplus _{n\in \mathbb{N}}p_{A_{n+1}}^{A_{n}}:\oplus _{n\in \mathbb{N}%
}A_{n}\rightarrow \oplus _{n\in \mathbb{N}}\frac{A_{n}}{A_{n+1}}
\end{equation*}%
is an algebra homomorphism. Moreover

\begin{enumerate}
\item[1)] $E=\oplus _{n\in \mathbb{N}}\frac{A_{n}}{A_{n+1}}$ is a graded
algebra such that $\oplus _{n\in \mathbb{N}}p_{A_{n+1}}^{A_{n}}$ is a graded
homomorphism;

\item[2)] $m_{a,b}^{E}\circ \left( p_{A_{a+1}}^{A_{a}}\otimes
p_{A_{b+1}}^{A_{b}}\right) =p_{A_{a+b+1}}^{A_{a+b}}\circ m_{a,b}^{A};$

\item[3)] $u_{E}=i_{0}^{E}\circ p_{A_{1}}^{A_{0}}\circ u_{0}^{A}.$
\end{enumerate}
\end{lemma}

\begin{proof}
It is similar to that of Lemma \ref{lem: GrDirect}.
\end{proof}

\begin{theorem}
\label{teo: strongly Alg}With hypothesis and notations of Lemma \ref{lem:
GrDirect}, the following assertions are equivalent.

\begin{enumerate}
\item[$\left( 1\right) $] $A=\oplus _{n\in \mathbb{N}}A_{n}$ is a strongly $%
\mathbb{N}
$-graded algebra.

\item[$\left( 2\right) $] $E=\oplus _{n\in \mathbb{N}}\frac{A_{n}}{A_{n-1}}$
is a strongly $%
\mathbb{N}
$-graded algebra.
\end{enumerate}
\end{theorem}

\begin{proof}
It is analogous to that of Theorem \ref{teo: strongly Alg co}.
\end{proof}

\begin{claim}
\label{claim: Ideal}Recall from \cite{AMS} that an \emph{ideal} of an
algebra $(A,m,u)$ in a monoidal category $(\mathcal{M},\otimes ,\mathbf{1})$
is a pair $(I,i_{I}^{A})$ where $I$ is an $A$-bimodule and%
\begin{equation*}
i_{I}^{A}:I\rightarrow A
\end{equation*}%
is a morphism of $A$-bimodules which is a monomorphism in $\mathcal{M}$.
\newline
A morphism $f:I\rightarrow J$ in $_{A}\mathcal{M}_{A}$, where $I$,$J$ are
two ideals, is called a \emph{morphism of ideals} whenever
\begin{equation*}
\diagMorphIdeal%
\end{equation*}%
Note that $f$ is a monomorphism in $\mathcal{M}$ as $i_{I}^{A}$ is a
monomorphism. Moreover $f$ is unique, as $i_{J}^{A}$ is a monomorphism.%
\newline
\end{claim}

\begin{claim}
Let $\mathcal{M}$ be an abelian monoidal category.\newline
Let $(I,i_{I}^{A})$ and $(J,i_{J}^{A})$ be two subobjects of an algebra $%
(A,m,u)$. Set
\begin{gather*}
m_{I,J}:=m(i_{I}^{A}\otimes i_{J}^{A}):I\otimes J\rightarrow A \\
(Q_{I,J},\pi _{I,J})=\text{Coker}\left( m_{I,J}\right) ,\qquad \pi
_{I,J}^{A}:A\rightarrow Q_{I,J} \\
(IJ,i_{IJ}^{A})=\mathrm{Ker}\left( \pi _{I,J}^{A}\right) =\text{\textrm{Im}}%
\left( m_{I,J}\right) ,\qquad i_{IJ}^{A}:IJ\rightarrow A
\end{gather*}%
The subobject $(IJ,i_{IJ}^{A})$ of $A$ is called \emph{the product of $I$
and $J$}.\newline
Moreover, we have the following exact sequence:%
\begin{equation}
0\longrightarrow IJ\overset{i_{IJ}^{A}}{\longrightarrow }A\overset{\pi
_{I,J}^{A}}{\longrightarrow }Q_{I,J}\longrightarrow 0.  \label{ec:IJ=Ker}
\end{equation}%
Since $(IJ,i_{IJ}^{A})=\mathrm{Ker}\left( \pi _{I,J}^{A}\right) \ $and $\pi
_{I,J}^{A}m_{I,J}=0$, by the universal property of the kernel, there is a
unique morphism $\overline{m}_{I,J}:I\otimes J\rightarrow IJ$ such that the
following diagram
\begin{equation*}
\diagMultIdeal%
\end{equation*}%
is commutative. Since $IJ=$\textrm{Im}$\left( m_{I,J}\right) ,$ it comes out
that $\overline{m}_{I,J}$ is an epimorphism.\newline
Consider the case $I=A.$\newline

Assume now that $(I,i_{I}^{A})$ and $(J,i_{J}^{A})$ are two ideals of $%
(A,m,u)$.

In this case $(IJ,i_{IJ}^{A})$ is an ideal of $A$ and $m_{I,J}\in {_{A}%
\mathcal{M}_{A}}$ so that it is straightforward to prove that $Q_{I,J}$ is
an algebra and that $\pi _{I,J}$ and algebra homomorphism.\newline
Since $i_{J}^{A}$ is a morphism in $_{A}\mathcal{M}$, we have
\begin{equation*}
m_{A,J}=m\circ (\mathrm{Id}_{A}\otimes i_{J}^{A})=i_{J}^{A}\circ \mu
_{J}^{l}.
\end{equation*}%
Since $i_{J}^{A}$ is a monomorphism and $\mu _{J}^{l}$ an epimorphism, we
deduce that%
\begin{equation*}
(AJ,iv_{AJ})=\text{\textrm{Im}}\left( m_{A,J}\right) =\left(
J,i_{J}^{A}\right) .
\end{equation*}%
Analogously, in the case $J=A,$ one has%
\begin{equation*}
(IA,i_{IA}^{A})=\left( I,i_{I}^{A}\right) .
\end{equation*}
\end{claim}

\begin{claim}
\label{claim: Ideal powers}Let $A$ be an algebra in an abelian monoidal
category $\mathcal{M}$ and let $(I,i_{I}^{A})$ be a subobject of $A$. We can
define (see \cite{AMS}) the $n$-th power $\left( I^{n},i_{I^{n}}^{A}\right) $
of $I,$ where $i_{I^{n}}^{A}:I^{n}\rightarrow A.$ By Definition
\begin{equation*}
I^{0}=A\qquad \text{and}\qquad I^{n}=I^{n-1}I,\text{ for every }n\geq 1.
\end{equation*}%
For every subobject $I,J,K$ of $A$ one can check that $\left( \left(
IJ\right) K,i_{\left( IJ\right) K}^{A}\right) $ and $\left( I\left(
JK\right) ,i_{I\left( JK\right) }^{A}\right) $ are isomorphic and thus can
be identified. 
Therefore $I^{i}I^{j}=I^{i+j}$ and we can consider $\overline{m}%
_{I^{i},I^{j}}:I^{i}\otimes I^{j}\rightarrow I^{i+j}.$ We set
\begin{equation*}
m_{i,j}^{I}:=\overline{m}_{I^{i},I^{j}}:I^{i}\otimes I^{j}\rightarrow
I^{i+j}.
\end{equation*}%
Assume now that $(I,i_{I}^{A})$ is an ideal of $A$. Then there is a (unique)
morphism of ideals%
\begin{equation*}
i_{I^{n+1}}^{I^{n}}:I^{n+1}\rightarrow I^{n},\text{ for every }n\in
\mathbb{N}
.
\end{equation*}%
Since $m$ is associative and by definition of $m_{i,j}^{I}$, it is
straightforward to prove that $m_{i,j}^{I}$ fulfills
\begin{gather}
m_{a+b,c}^{I}\left( m_{a,b}^{I}\otimes I^{c}\right) =m_{a,b+c}^{I}\left(
I^{a}\otimes m_{b,c}^{I}\right) .  \label{form: locMultiIdeal} \\
m_{d,0}^{I}\left( I^{d}\otimes u\right) =r_{I^{d}},\qquad m_{0,d}^{I}\left(
u\otimes I^{d}\right) =l_{I^{d}}  \label{form: locUnitIdeal} \\
m_{a,b}^{I}\circ \left( i_{I^{a+1}}^{I^{a}}\otimes I^{b}\right)
=i_{I^{a+b+1}}^{I^{a+b}}\circ m_{a+1,b}^{I},\qquad m_{a,b}^{I}\circ \left(
I^{a}\otimes i_{I^{b+1}}^{I^{b}}\right) =i_{I^{a+b+1}}^{I^{a+b}}\circ
m_{a,b+1}^{I}  \label{form: locComp algIdeal}
\end{gather}
\end{claim}

\begin{theorem}
\label{teo: gr}Let $\mathcal{M}$ be a cocomplete abelian monoidal category
such that the tensor product commutes with direct sums.

Let $\left( A,m,u\right) $ be an algebra in $\mathcal{M}$ and let $\left(
I,i_{I}\right) $ be an ideal of $A$ in $\mathcal{M}$. For every $n\in
\mathbb{N}
$, we set
\begin{equation*}
gr_{I}^{n}A=\frac{I^{n}}{I^{n+1}}.
\end{equation*}%
Then $\oplus _{n\in \mathbb{N}}I^{n}$ is a graded algebra and there are
unique algebra structure on $gr_{I}A:=\oplus _{n\in \mathbb{N}}gr_{I}^{n}A$
such that
\begin{equation*}
\oplus _{n\in \mathbb{N}}p_{I^{n+1}}^{I^{n}}:\oplus _{n\in \mathbb{N}%
}I^{n}\rightarrow gr_{I}A
\end{equation*}%
is an algebra homomorphism. Moreover

\begin{enumerate}
\item[1)] $gr_{I}A$ is a graded algebra such that $\oplus _{n\in \mathbb{N}%
}\pi _{n}^{I}$ is a graded homomorphism;

\item[2)]
\begin{equation}
m_{a,b}^{gr_{I}A}\circ \left( p_{I^{a+1}}^{I^{a}}\otimes
p_{I^{b+1}}^{I^{b}}\right) =p_{I^{a+b+1}}^{I^{a+b}}\circ m_{a,b}^{I},\text{
for every }a,b\in
\mathbb{N}
\label{form: multi gr_IE}
\end{equation}

\item[3)] $u_{gr_{I}A}=i_{0}^{gr_{I}A}\circ p_{I}^{A}\circ u$
\end{enumerate}

Moreover $gr_{I}A$ is a strongly $%
\mathbb{N}
$-graded algebra.
\end{theorem}

\begin{proof}
By (\ref{form: locMultiIdeal}), (\ref{form: locUnitIdeal}) and (\ref{form:
locComp algIdeal}), we can apply Lemma \ref{lem: GrDirect Inverse} to the
family $(I^{n})_{n\in
\mathbb{N}
}.$ It remains to prove the last assertion. From 2), since both $%
p_{I^{a+b+1}}^{I^{a+b}}$ and $m_{a,b}^{I}$ are epimorphisms, we get that $%
m_{a,b}^{gr_{I}A}$ is an epimorphism too, for every $a,b\in \mathbb{N}.$
Thus $gr_{I}A$ is a strongly $%
\mathbb{N}
$-graded algebra.
\end{proof}

\begin{definition}
Let $\mathcal{M}$ be a cocomplete abelian monoidal category such that the
tensor product commutes with direct sums.

Given an ideal $I$ of an algebra $A$ in $\mathcal{M}$, the strongly $%
\mathbb{N}
$-graded algebra $gr_{I}A$ defined in Theorem \ref{teo: gr} will be called
the \emph{associated graded algebra }(of $A$ with respect to $I$).
\end{definition}

\begin{theorem}
\label{teo: graded}Let $\mathcal{M}$ be a cocomplete abelian monoidal
category such that the tensor product commutes with direct sums. Let $I$ be
an ideal of an algebra $A$ in $\mathcal{M}$ an let $gr_{I}A$ be the
associated graded algebra.

Let
\begin{equation*}
T:=T_{\frac{A}{I}}(\frac{I}{I^{2}})
\end{equation*}
be the tensor algebra. Then there is a unique algebra homomorphism
\begin{equation*}
\varphi :T_{\frac{A}{I}}(\frac{I}{I^{2}})\rightarrow gr_{I}A,
\end{equation*}%
such that $\varphi \circ i_{0}^{T}=i_{0}^{gr_{I}A}$ and $\varphi \circ
i_{1}^{T}=i_{1}^{gr_{I}A}$. \newline
Moreover $\varphi $ is a graded algebra homomorphism with%
\begin{equation*}
\varphi _{t}=p_{t}^{gr_{I}A}\circ \overline{m}_{gr_{I}A}^{t-1}\circ \left(
i_{1}^{gr_{I}A}\right) ^{\otimes _{A_{0}}t}\text{ for every }t\in
\mathbb{N}%
\end{equation*}%
and the following equivalent assertions hold.

$\left( a\right) $ $gr_{I}A$ is a strongly $%
\mathbb{N}
$-graded algebra.

$\left( a^{\prime }\right) $ $m_{a,1}^{gr_{I}A}:gr_{I}^{a}A\otimes
gr_{I}^{1}A\rightarrow gr_{I}^{a+1}A$ is an epimorphism for every $a\in
\mathbb{N}
$.

$\left( b\right) \ \varphi _{n}$ is an epimorphism for every $n\in
\mathbb{N}
$.

$\left( c\right) $ $\varphi $ is an epimorphism.

$\left( d\right) $ $\oplus _{i\geq n}gr_{I}^{i}A=\left( \oplus _{i\geq
1}gr_{I}^{i}A\right) ^{n},$ for every $n\in \mathbb{N}.$

$\left( e\right) $ $\oplus _{i\geq 2}gr_{I}^{i}A=\left( \oplus _{i\geq
1}gr_{I}^{i}A\right) ^{2}$.
\end{theorem}

\begin{proof}
By Theorem \ref{teo: gr}, $(a)$ holds. We conclude by applying \cite[Theorem
3.11]{AM}.
\end{proof}

\section{The Associated Graded Coalgebra of a Bialgebra With Respect to a
Subbialgebra \label{sec: ass Grad Coalg Bialg}}

\begin{lemma}
\label{lem: graded braided}Let $\left( \mathcal{M},c\right) $ be a
cocomplete abelian braided monoidal category such that the tensor product
commutes with direct sums. Consider in $\left( \mathcal{M},c\right) $ a
datum $(B,m_{B},u_{B},\Delta _{B},\varepsilon _{B})$ consisting of a graded
object $B$ with graduation defined by $\left( B_{k}\right) _{k\in \mathbb{N}%
} $ such that, with respect to this graduation,

\begin{itemize}
\item $\left( B,m_{B},u_{B}\right) $ is a graded algebra and

\item $(B,\Delta _{B},\varepsilon _{B})$ is a graded coalgebra.
\end{itemize}

Assume that for every $a,b\in
\mathbb{N}
,$%
\begin{eqnarray*}
&&\sum\limits_{s+t=a+b}\left( i_{s}\otimes i_{t}\right) \Delta _{s,t}m_{a,b}
\\
&=&\sum\limits_{s^{\prime }+t^{\prime }=a}\sum\limits_{s^{\prime \prime
}+t^{\prime \prime }=b}\left( i_{s^{\prime }+s^{\prime \prime }}\otimes
i_{t^{\prime }+t^{\prime \prime }}\right) \circ \left( m_{s^{\prime
},s^{\prime \prime }}\otimes m_{t^{\prime },t^{\prime \prime }}\right) \circ
\left( B_{s^{\prime }}\otimes c_{B_{t^{\prime }},B_{s^{\prime \prime
}}}\otimes B_{t^{\prime \prime }}\right) \circ \left( \Delta _{s^{\prime
},t^{\prime }}\otimes \Delta _{s^{\prime \prime },t^{\prime \prime }}\right)
\end{eqnarray*}%
\begin{equation*}
\varepsilon _{0}m_{0,0}=m_{\mathbf{1}}\left( \varepsilon _{0}\otimes
\varepsilon _{0}\right) .
\end{equation*}%
Then $B$ is a graded braided bialgebra in $\left( \mathcal{M},c\right) $.
\end{lemma}

\begin{proof}
It follows easily by using (\ref{form: coro grAlg1}), (\ref{form: grCoalg}).
\end{proof}

\begin{theorem}
\label{pro: Sweedler}Let $\mathcal{M}$ be an coabelian monoidal category.
Let $((X_{i})_{i\in \mathbb{N}},(\xi _{i}^{j})_{i,j\in \mathbb{N}})$ be a
direct system in $\mathcal{M}$ where, for $i\leq j$, $\xi
_{i}^{j}:X_{i}\rightarrow X_{j}$.

Let $(\xi _{i}:X_{i}\rightarrow X)_{i\in \mathbb{N}}$ be a compatible family
of morphisms with respect to the given direct system.

Assume that

\begin{itemize}
\item $\xi _{i}^{j}:X_{i}\rightarrow X_{j}$ is a split monomorphism for
every $i\leq j,$

\item $X_{0}=0,$

\item $\xi _{i}:X_{i}\rightarrow X$ is a monomorphism for every $i\in
\mathbb{N}
$
\end{itemize}

and denote by $\tau _{i}:X\rightarrow \frac{X}{X_{i}}$ the canonical
projection for every $i\in
\mathbb{N}
.$

Then, for every $n\in
\mathbb{N}
,$ the following sequence is exact.%
\begin{equation*}
\bigoplus\limits_{a+b=n+1}X_{a}\otimes X_{b}\overset{\nabla \left[ \left(
\xi _{a}\otimes \xi _{b}\right) _{a+b=n+1}\right] }{\longrightarrow }%
X\otimes X\overset{\Delta \left[ \left( \tau _{a}\otimes \tau _{b}\right)
_{a+b=n}\right] }{\longrightarrow }\bigoplus\limits_{a+b=n}\frac{X}{X_{a}}%
\otimes \frac{X}{X_{b}}.
\end{equation*}
\end{theorem}

\begin{proof}
Apply \cite[Theorem 3.1]{AM2}.
\end{proof}

\begin{notations}
In this section, from now on, the following assumptions and notations will
be used.

$\left( \mathcal{M},c\right) $ is a cocomplete abelian coabelian braided
monoidal category such that the tensor product commutes with direct sums.
\newline
Let $i_{B}^{E}:B\hookrightarrow E$ be a monomorphism in $\mathcal{M}$ which
is a braided bialgebra homomorphism in $\mathcal{M}$ and let $i_{B^{\wedge
_{E}n}}^{E}:B^{\wedge _{E}n}\rightarrow E$ and $i_{B^{\wedge
_{E}a}}^{B^{\wedge _{E}b}}:B^{\wedge _{E}a}\rightarrow B^{\wedge _{E}b}$ ($%
a\leq b$) be the canonical injections.

Assume that $i_{B^{\wedge _{E}a}}^{B^{\wedge _{E}b}}$ is a split
monomorphism in $\mathcal{M}$ for every $a\leq b$.

By Theorem \ref{pro: Sweedler}, we have the following exact sequence%
\begin{equation*}
\bigoplus\limits_{a+b=n+1}B^{\wedge _{E}a}\otimes B^{\wedge _{E}b}\overset{%
\nabla \left[ \left( i_{B^{\wedge _{E}a}}^{E}\otimes i_{B^{\wedge
_{E}b}}^{E}\right) _{a+b=n+1}\right] }{\longrightarrow }E\otimes E\overset{%
\Delta \left[ \left( p_{B^{\wedge _{E}a}}^{E}\otimes p_{B^{\wedge
_{E}b}}^{E}\right) _{a+b=n}\right] }{\longrightarrow }\bigoplus%
\limits_{a+b=n}\frac{E}{B^{\wedge _{E}a}}\otimes \frac{E}{B^{\wedge _{E}b}}.
\end{equation*}%
Let%
\begin{equation*}
\left( \sum_{a+b=n+1}B^{\wedge _{E}a}\otimes B^{\wedge _{E}b},\beta
_{n}\right) =\mathrm{Im}\left\{ \nabla \left[ \left( i_{B^{\wedge
_{E}a}}^{E}\otimes i_{B^{\wedge _{E}b}}^{E}\right) _{a+b=n+1}\right] \right\}
\end{equation*}%
and let
\begin{equation*}
\gamma _{n}:\bigoplus\limits_{a+b=n+1}B^{\wedge _{E}a}\otimes B^{\wedge
_{E}b}\twoheadrightarrow \sum_{a+b=n+1}B^{\wedge _{E}a}\otimes B^{\wedge
_{E}b}
\end{equation*}%
be the unique morphism such that
\begin{equation}
\beta _{n}\circ \gamma _{n}=\nabla \left[ \left( i_{B^{\wedge
_{E}a}}^{E}\otimes i_{B^{\wedge _{E}b}}^{E}\right) _{a+b=n+1}\right] .
\label{form: def gamma}
\end{equation}%
Since $\left( B^{\wedge _{E}a+b},i_{B^{\wedge _{E}a+b}}^{E}\right) =\ker
\left( \left( p_{B^{\wedge _{E}a}}^{E}\otimes p_{B^{\wedge
_{E}b}}^{E}\right) \Delta _{E}\right) ,$ we have $\left( p_{B^{\wedge
_{E}a}}^{E}\otimes p_{B^{\wedge _{E}b}}^{E}\right) \circ \Delta _{E}\circ
i_{B^{\wedge _{E}a+b}}^{E}=0$ so that%
\begin{equation*}
\Delta \left[ \left( p_{B^{\wedge _{E}a}}^{E}\otimes p_{B^{\wedge
_{E}b}}^{E}\right) _{a+b=n}\right] \circ \Delta _{E}\circ i_{B^{\wedge
_{E}a+b}}^{E}=0
\end{equation*}%
and hence, by the exactness of the sequence above, there exists a unique
morphism
\begin{equation*}
\alpha _{n}:B^{\wedge _{E}n}\rightarrow \sum_{a+b=n+1}B^{\wedge
_{E}a}\otimes B^{\wedge _{E}b}
\end{equation*}%
such that%
\begin{equation}
\Delta _{E}\circ i_{B^{\wedge _{E}n}}^{E}=\beta _{n}\circ \alpha _{n},\text{
for every }n\in
\mathbb{N}
\text{.}  \label{form: def alpha}
\end{equation}
\end{notations}

\begin{lemma}
\label{lem: mLambda}0) For every $s,t,u,v\in
\mathbb{N}
,$ we have
\begin{eqnarray}
&&\left( p_{B^{\wedge _{E}s}}^{E}\otimes p_{B^{\wedge _{E}t}}^{E}\right)
\circ \left( m_{E}\otimes m_{E}\right) \circ \left( E\otimes c\otimes
E\right) \circ \left( \beta _{u}\otimes \beta _{v}\right) \circ \left(
\gamma _{u}\otimes \gamma _{v}\right)  \label{form: Sweed5,5} \\
&=&\nabla \left[ \left(
\begin{array}{c}
\left( \left[ p_{B^{\wedge _{E}s}}^{E}m_{E}\left( i_{B^{\wedge
_{E}a}}^{E}\otimes i_{B^{\wedge _{E}b}}^{E}\right) \right] \otimes \left[
p_{B^{\wedge _{E}t}}^{E}m_{E}\left( i_{B^{\wedge _{E}c}}^{E}\otimes
i_{B^{\wedge _{E}d}}^{E}\right) \right] \right) \\
\circ \left( B^{\wedge _{E}a}\otimes c_{B^{\wedge _{E}c},B^{\wedge
_{E}b}}\otimes B^{\wedge _{E}d}\right)%
\end{array}%
\right) _{\substack{ a+c=u+1  \\ b+d=v+1}}\right]  \notag
\end{eqnarray}

1) The following relations hold.%
\begin{equation}
p_{B^{\wedge _{E}u+v-1}}^{E}m_{E}\left( i_{B^{\wedge _{E}u}}^{E}\otimes
i_{B^{\wedge _{E}v}}^{E}\right) =0\text{ for every }u,v\in
\mathbb{N}
,u+v\geq 1.  \label{form: Sweed6,5}
\end{equation}%
2) For every $a,b\in
\mathbb{N}
$, there exists a unique morphism $m_{\wedge }^{a,b}:B^{\wedge
_{E}a+1}\otimes B^{\wedge _{E}b+1}\rightarrow B^{\wedge _{E}a+b+1}$ such
that
\begin{equation}
i_{B^{\wedge _{E}a+b+1}}^{E}\circ m_{\wedge }^{a,b}=m_{E}\left( i_{B^{\wedge
_{E}a+1}}^{E}\otimes i_{B^{\wedge _{E}b+1}}^{E}\right) ,\text{ for every }%
a,b\in
\mathbb{N}
.  \label{form: Sweed7,5}
\end{equation}%
3) For every $a,b,c,d\in
\mathbb{N}
$, we have%
\begin{gather*}
m_{\wedge }^{a+b,c}\circ \left( m_{\wedge }^{a,b}\otimes B^{\wedge
_{E}c+1}\right) =m_{\wedge }^{a,b+c}\circ \left( B^{\wedge _{E}a+1}\otimes
m_{\wedge }^{b,c}\right) \text{,} \\
m_{\wedge }^{d,0}\circ \left( B^{\wedge _{E}d+1}\otimes u_{B}\right)
=r_{B^{\wedge _{E}d+1}},\qquad m_{\wedge }^{0,d}\circ \left( u_{B}\otimes
B^{\wedge _{E}d+1}\right) =l_{B^{\wedge _{E}d+1}}\text{,} \\
m_{\wedge }^{a+1,b}\circ \left( i_{B^{\wedge _{E}a+1}}^{B^{\wedge
_{E}a+2}}\otimes B^{\wedge _{E}b+1}\right) =i_{B^{\wedge
_{E}a+b+1}}^{B^{\wedge _{E}a+b+2}}\circ m_{\wedge }^{a,b},\qquad m_{\wedge
}^{a,b+1}\circ \left( B^{\wedge _{E}a+1}\otimes i_{B^{\wedge
_{E}b+1}}^{B^{\wedge _{E}b+2}}\right) =i_{B^{\wedge _{E}a+b+1}}^{B^{\wedge
_{E}a+b+2}}\circ m_{\wedge }^{a,b}
\end{gather*}
\end{lemma}

\begin{proof}
See the Appendix.
\end{proof}

\begin{proposition}
\label{pro: grE alg}$\oplus _{n\in \mathbb{N}}B^{\wedge _{E}n+1}$ is a
graded algebra and there are unique algebra structure on $gr_{B}E=\oplus
_{n\in \mathbb{N}}gr_{B}^{n}E$ such that
\begin{equation*}
\oplus _{n\in \mathbb{N}}p_{B^{\wedge _{E}n}}^{B^{\wedge _{E}n+1}}:\oplus
_{n\in \mathbb{N}}B^{\wedge _{E}n+1}\rightarrow gr_{B}E
\end{equation*}%
is an algebra homomorphism. Moreover

\begin{enumerate}
\item[1)] $gr_{B}E=\oplus _{n\in \mathbb{N}}gr_{B}^{n}E$ is a graded algebra
such that $\oplus _{n\in \mathbb{N}}p_{B^{\wedge _{E}n}}^{B^{\wedge
_{E}n+1}} $ is a graded homomorphism;

\item[2)]
\begin{equation}
p_{B^{\wedge _{E}a+b}}^{B^{\wedge _{E}a+b+1}}\circ m_{\wedge
}^{a,b}=m_{a,b}^{gr_{B}E}\circ \left( p_{B^{\wedge _{E}a}}^{B^{\wedge
_{E}a+1}}\otimes p_{B^{\wedge _{E}b}}^{B^{\wedge _{E}b+1}}\right) .
\label{form: Sweed8}
\end{equation}

\item[3)] $u_{gr_{B}E}=i_{0}^{gr_{B}E}\circ p_{B^{\wedge _{E}0}}^{B}\circ
u_{B}$
\end{enumerate}
\end{proposition}

\begin{proof}
By Lemma \ref{lem: mLambda}, we can apply Lemma \ref{lem: GrDirect} to the
family $(B^{\wedge _{E}n+1})_{n\in
\mathbb{N}
}.$
\end{proof}

\begin{theorem}
\label{teo: gr_BE Bialg}Let $\left( \mathcal{M},c\right) $ be a cocomplete
and complete abelian coabelian braided monoidal category satisfying $AB5$.
Assume that the tensor product commutes with direct sums. Let $%
i_{B}^{E}:B\hookrightarrow E$ be a monomorphism in $\mathcal{M}$ which is a
braided bialgebra homomorphism in $\mathcal{M}$ and let $i_{B^{\wedge
_{E}n}}^{E}:B^{\wedge _{E}n}\rightarrow E$ and $i_{B^{\wedge
_{E}a}}^{B^{\wedge _{E}b}}:B^{\wedge _{E}a}\rightarrow B^{\wedge _{E}b}$ ($%
a\leq b$) be the canonical injections. Assume that $i_{B^{\wedge
_{E}a}}^{B^{\wedge _{E}b}}$ is a split monomorphism in $\mathcal{M}$ for
every $a\leq b$.\newline
Then $gr_{B}E$ is a graded braided bialgebra in $\left( \mathcal{M},c\right)
$.
\end{theorem}

\begin{proof}
Set $i_{n}:=i_{B^{\wedge _{E}a}}^{E}$, $p_{n}:=p_{B^{\wedge _{E}n}}^{E}$ and
$p_{n}^{n+1}:=p_{B^{\wedge _{E}n}}^{B^{\wedge _{E}n+1}}.$

By Theorem \ref{teo: gr dual}, $\left( gr_{B}E,\Delta _{gr_{B}E},\varepsilon
_{gr_{B}E}=\varepsilon _{C}\circ \sigma _{C}^{0}\circ p_{0}^{gr_{C}E}\right)
$ is a strongly $%
\mathbb{N}
$-graded coalgebra.

By Proposition \ref{pro: grE alg}, $\left(
gr_{B}E,m_{gr_{B}E},u_{gr_{B}E}=i_{0}^{gr_{B}E}\circ p_{0}^{1}\circ
u_{B}\right) $ is a graded algebra with the same graduation defined by $%
\left( gr_{B}^{k}E\right) _{k\in \mathbb{N}}$.

By Lemma \ref{lem: graded braided}, in order to get that $gr_{B}E$ is a
graded braided bialgebra we have to prove that%
\begin{eqnarray*}
&&\sum\limits_{s+t=a+b}\left( i_{s}^{gr_{B}E}\otimes i_{t}^{gr_{B}E}\right)
\circ \Delta _{s,t}^{gr_{B}E}\circ m_{a,b}^{gr_{B}E} \\
&=&\sum\limits_{s^{\prime }+t^{\prime }=a}\sum\limits_{s^{\prime \prime
}+t^{\prime \prime }=b}\left[
\begin{array}{c}
\left( i_{s^{\prime }+s^{\prime \prime }}^{gr_{B}E}\otimes i_{t^{\prime
}+t^{\prime \prime }}^{gr_{B}E}\right) \circ \left( m_{s^{\prime },s^{\prime
\prime }}^{gr_{B}E}\otimes m_{t^{\prime },t^{\prime \prime
}}^{gr_{B}E}\right) \circ \\
\circ \left( \frac{B^{\wedge _{E}s^{\prime }+1}}{B^{\wedge _{E}s^{\prime }}}%
\otimes c_{\frac{B^{\wedge _{E}t^{\prime }+1}}{B^{\wedge _{E}t^{\prime }}},%
\frac{B^{\wedge _{E}s^{\prime \prime }+1}}{B^{\wedge _{E}s^{\prime \prime }}}%
}\otimes \frac{B^{\wedge _{E}t^{\prime \prime }+1}}{B^{\wedge _{E}t^{\prime
\prime }}}\right) \circ \left( \Delta _{s^{\prime },t^{\prime
}}^{gr_{B}E}\otimes \Delta _{s^{\prime \prime },t^{\prime \prime
}}^{gr_{B}E}\right)%
\end{array}%
\right] ,
\end{eqnarray*}%
for every $a,b\in
\mathbb{N}
.$ Denote by%
\begin{equation*}
j_{z}:\frac{E}{B^{\wedge _{E}z}}\rightarrow \oplus _{i\in
\mathbb{N}
}\frac{E}{B^{\wedge _{E}i}}
\end{equation*}%
the canonical injection. Since%
\begin{equation*}
\left( \oplus _{w\in
\mathbb{N}
}\frac{i_{w+1}}{B^{\wedge _{E}w}}\right) \circ i_{s}^{gr_{B}E}=j_{s}\circ
\frac{i_{s+1}}{B^{\wedge _{E}s}}\text{ for every }s\in
\mathbb{N}%
\end{equation*}%
and $\oplus _{w\in
\mathbb{N}
}\frac{i_{w+1}}{B^{\wedge _{E}w}}$ is a monomorphism (our category is an $%
AB4 $ category \cite[page 53]{Po}) and $p_{a}^{a+1}\otimes p_{b}^{b+1}$ is
an epimorphism, the equality we have to prove is equivalent to
\begin{eqnarray*}
&&\sum\limits_{s+t=a+b}\left( j_{s}\otimes j_{t}\right) \circ \left( \frac{%
i_{s+1}}{B^{\wedge _{E}s}}\otimes \frac{i_{t+1}}{B^{\wedge _{E}t}}\right)
\circ \Delta _{s,t}^{gr_{B}E}\circ m_{a,b}^{gr_{B}E}\circ \left(
p_{a}^{a+1}\otimes p_{b}^{b+1}\right) \\
&=&\sum\limits_{s^{\prime }+t^{\prime }=a}\sum\limits_{s^{\prime \prime
}+t^{\prime \prime }=b}\left[
\begin{array}{c}
\left( j_{s^{\prime }+s^{\prime \prime }}\otimes j_{t^{\prime }+t^{\prime
\prime }}\right) \circ \left( \frac{i_{s^{\prime }+s^{\prime \prime }+1}}{%
B^{\wedge _{E}s^{\prime }+s^{\prime \prime }}}\otimes \frac{i_{t^{\prime
}+t^{\prime \prime }+1}}{B^{\wedge _{E}t^{\prime }+t^{\prime \prime }}}%
\right) \circ \left( m_{s^{\prime },s^{\prime \prime }}^{gr_{B}E}\otimes
m_{t^{\prime },t^{\prime \prime }}^{gr_{B}E}\right) \circ \\
\circ \left( \frac{B^{\wedge _{E}s^{\prime }+1}}{B^{\wedge _{E}s^{\prime }}}%
\otimes c_{\frac{B^{\wedge _{E}t^{\prime }+1}}{B^{\wedge _{E}t^{\prime }}},%
\frac{B^{\wedge _{E}s^{\prime \prime }+1}}{B^{\wedge _{E}s^{\prime \prime }}}%
}\otimes \frac{B^{\wedge _{E}t^{\prime \prime }+1}}{B^{\wedge _{E}t^{\prime
\prime }}}\right) \\
\circ \left( \Delta _{s^{\prime },t^{\prime }}^{gr_{B}E}\otimes \Delta
_{s^{\prime \prime },t^{\prime \prime }}^{gr_{B}E}\right) \circ \left(
p_{a}^{a+1}\otimes p_{b}^{b+1}\right)%
\end{array}%
\right]
\end{eqnarray*}

By using (\ref{form: Sweed8}), (\ref{form: Delta grCo}), (\ref{form:
Sweed7,5}), the compatibility between $\Delta _{E}$ and $m_{E},$ and (\ref%
{form: def alpha}), the first term rewrites as%
\begin{equation*}
\sum\limits_{s+t=a+b}\left( j_{s}\otimes j_{t}\right) \circ \left(
p_{s}\otimes p_{t}\right) \circ \left( m_{E}\otimes m_{E}\right) \circ
\left( E\otimes c_{E,E}\otimes E\right) \circ \left( \beta _{a+1}\otimes
\beta _{b+1}\right) \circ \left( \alpha _{a+1}\otimes \alpha _{b+1}\right)
\end{equation*}

On the other hand, in view of (\ref{form: def theta}), the second term
rewrites as $\Xi \circ \left( \alpha _{a+1}\otimes \alpha _{b+1}\right) $
where%
\begin{equation*}
\Xi =\sum\limits_{\substack{ s^{\prime }+t^{\prime }=a  \\ s^{\prime \prime
}+t^{\prime \prime }=b}}\left[
\begin{array}{c}
\left( j_{s^{\prime }+s^{\prime \prime }}\otimes j_{t^{\prime }+t^{\prime
\prime }}\right) \circ \left( \frac{i_{s^{\prime }+s^{\prime \prime }+1}}{%
B^{\wedge _{E}s^{\prime }+s^{\prime \prime }}}\otimes \frac{i_{t^{\prime
}+t^{\prime \prime }+1}}{B^{\wedge _{E}t^{\prime }+t^{\prime \prime }}}%
\right) \circ \left( m_{s^{\prime },s^{\prime \prime }}^{gr_{B}E}\otimes
m_{t^{\prime },t^{\prime \prime }}^{gr_{B}E}\right) \circ \\
\circ \left( \frac{B^{\wedge _{E}s^{\prime }+1}}{B^{\wedge _{E}s^{\prime }}}%
\otimes c_{\frac{B^{\wedge _{E}t^{\prime }+1}}{B^{\wedge _{E}t^{\prime }}},%
\frac{B^{\wedge _{E}s^{\prime \prime }+1}}{B^{\wedge _{E}s^{\prime \prime }}}%
}\otimes \frac{B^{\wedge _{E}t^{\prime \prime }+1}}{B^{\wedge _{E}t^{\prime
\prime }}}\right) \circ \left( \theta _{s^{\prime },t^{\prime }}\otimes
\theta _{s^{\prime \prime },t^{\prime \prime }+1}\right)%
\end{array}%
\right] .
\end{equation*}

We will prove that%
\begin{equation*}
\Xi =\sum\limits_{s+t=a+b}\left( j_{s}\otimes j_{t}\right) \circ \left(
p_{s}\otimes p_{t}\right) \circ \left( m_{E}\otimes m_{E}\right) \circ
\left( E\otimes c_{E,E}\otimes E\right) \circ \left( \beta _{a+1}\otimes
\beta _{b+1}\right)
\end{equation*}

Since $\gamma _{a+1}\otimes \gamma _{b+1}$ is an epimorphism, equivalently
we will prove that%
\begin{eqnarray*}
&&\Xi \circ \left( \gamma _{a+1}\otimes \gamma _{b+1}\right) \\
&=&\sum\limits_{s+t=a+b}\left( j_{s}\otimes j_{t}\right) \circ \left(
p_{s}\otimes p_{t}\right) \circ \left( m_{E}\otimes m_{E}\right) \circ
\left( E\otimes c_{E,E}\otimes E\right) \circ \left( \beta _{a+1}\otimes
\beta _{b+1}\right) \circ \left( \gamma _{a+1}\otimes \gamma _{b+1}\right) .
\end{eqnarray*}

This is achieved by using (\ref{form: teta2}), (\ref{form: Sweed8}), (\ref%
{form: Sweed7,5}), naturality of braiding and (\ref{form: Sweed5,5}).

In view of Lemma \ref{lem: graded braided}, it remains to prove that%
\begin{equation*}
\varepsilon _{0}^{gr_{B}E}\circ m_{0,0}^{gr_{B}E}=m_{\mathbf{1}}\circ \left(
\varepsilon _{0}^{gr_{B}E}\otimes \varepsilon _{0}^{gr_{B}E}\right) .
\end{equation*}%
This follows easily once proved that $\varepsilon _{0}^{gr_{B}E}=\varepsilon
_{B}$ and $m_{0,0}^{gr_{B}E}=m_{B}.$

In view of Theorem \ref{teo: gr dual}, we have%
\begin{equation*}
\varepsilon _{0}^{gr_{B}E}=\varepsilon _{gr_{C}E}\circ
i_{0}^{gr_{B}E}=\varepsilon _{E}\circ \sigma _{B}^{0}\circ
p_{0}^{gr_{B}E}\circ i_{0}^{gr_{B}E}=\varepsilon _{E}\circ
i_{B}^{E}=\varepsilon _{B}.
\end{equation*}%
Since%
\begin{equation*}
i_{B}^{E}\circ m_{\wedge }^{0,0}\overset{\text{(\ref{form: Sweed7,5})}}{=}%
m_{E}\left( i_{B}^{E}\otimes i_{B}^{E}\right) =i_{B}^{E}m_{B}
\end{equation*}%
we get that $m_{\wedge }^{0,0}=m_{B}$ and hence%
\begin{equation*}
m_{0,0}^{gr_{B}E}=m_{0,0}^{gr_{B}E}\circ \left( p_{0}^{1}\otimes
p_{0}^{1}\right) \overset{\text{(\ref{form: Sweed8})}}{=}p_{0}^{1}\circ
m_{\wedge }^{0,0}=m_{\wedge }^{0,0}=m_{B}.
\end{equation*}
\end{proof}

\begin{theorem}
\label{teo: typeOne Alg}Let $\left( \mathcal{M},c\right) $ be a cocomplete
and complete abelian coabelian braided monoidal category satisfying $AB5$.
Assume that the tensor product commutes with direct sums. Let $%
i_{B}^{E}:B\hookrightarrow E$ be a monomorphism in $\mathcal{M}$ which is a
braided bialgebra homomorphism in $\mathcal{M}$ and let $i_{B^{\wedge
_{E}n}}^{E}:B^{\wedge _{E}n}\rightarrow E$ and $i_{B^{\wedge
_{E}a}}^{B^{\wedge _{E}b}}:B^{\wedge _{E}a}\rightarrow B^{\wedge _{E}b}$ ($%
a\leq b$) be the canonical injections. Assume that $i_{B^{\wedge
_{E}a}}^{B^{\wedge _{E}b}}$ is a split monomorphism in $\mathcal{M}$ for
every $a\leq b$. The following assertions are equivalent.

\begin{enumerate}
\item[$\left( 1\right) $] $gr_{B}E$ is the braided bialgebra of type one
associated to $B$ and $\frac{B\wedge _{E}B}{B}$.

\item[$\left( 2\right) $] $gr_{B}E$ is strongly $%
\mathbb{N}
$-graded as an algebra.

\item[$\left( 3\right) $] $A=\oplus _{n\in \mathbb{N}}B^{\wedge _{E}n+1}$ is
strongly $%
\mathbb{N}
$-graded as an algebra.

\item[$\left( 4\right) $] $B^{\wedge _{E}n+1}=\left( B^{\wedge _{E}2}\right)
^{\cdot _{E}n}$ for every $n\geq 2.$
\end{enumerate}
\end{theorem}

\begin{proof}
Consider the graded algebra homomorphism
\begin{equation*}
\oplus _{n\in \mathbb{N}}p_{B^{\wedge _{E}n}}^{B^{\wedge _{E}n+1}}:\oplus
_{n\in \mathbb{N}}B^{\wedge _{E}n+1}\rightarrow gr_{B}E
\end{equation*}%
of Proposition \ref{pro: grE alg}.

$\left( 1\right) \Leftrightarrow \left( 2\right) $ It follows in view of
\cite[Theorem 6.8]{AM} (where $AB5$ is required) and by Theorem \ref{teo: gr
dual}.

$\left( 2\right) \Leftrightarrow \left( 3\right) $ It follows by Theorem \ref%
{teo: strongly Alg}.

$\left( 3\right) \Leftrightarrow \left( 4\right) $ Let $\varphi
:T=T_{B}\left( B^{\wedge _{E}2}\right) \rightarrow \oplus _{n\in \mathbb{N}%
}B^{\wedge _{E}n+1}$ be the canonical morphism arising from the universal
property of the tensor algebra and let $\varphi _{n}:\left( B^{\wedge
_{E}2}\right) ^{\otimes _{B}n}\rightarrow B^{\wedge _{E}n+1}$ be its graded $%
n$-th component. In view of \cite[Theorem 3.11]{AM}, $\left( 3\right) $ is
equivalent to require that $\varphi _{n}$ is an epimorphism for every $n\geq
2$ (note that $\varphi _{0}$ and $\varphi _{1}$ are always isomorphisms).
Let us prove that%
\begin{equation}
\varphi _{n}\circ \chi _{\left( B^{\wedge _{E}2}\right) ^{\otimes
_{B}n-1},B^{\wedge _{E}2}}=m_{\wedge }^{n-1,1}\circ \left( \varphi
_{n-1}\otimes B^{\wedge _{E}2}\right) ,\text{ for every }n\geq 2.
\label{form: phi}
\end{equation}%
where $\chi _{X,Y}:X\otimes Y\rightarrow X\otimes _{B}Y$ denotes the
canonical projection.

Note that, being $\varphi $ a graded homomorphism and in view of its
definition, one has%
\begin{equation*}
i_{B^{\wedge _{E}n}}^{A}\circ \varphi _{n-1}=\varphi \circ i_{\left(
B^{\wedge _{E}2}\right) ^{\otimes _{B}n-1}}^{T}=\overline{m}_{A}^{n-2}\circ
\left( i_{B^{\wedge _{E}2}}^{A}\right) ^{\otimes _{B}n-1}
\end{equation*}%
so that%
\begin{eqnarray*}
&&\varphi _{n}\circ \chi _{\left( B^{\wedge _{E}2}\right) ^{\otimes
_{B}n-1},B^{\wedge _{E}2}}=p_{n}^{A}\circ \overline{m}_{A}^{n-1}\circ \left(
i_{B^{\wedge _{E}2}}^{A}\right) ^{\otimes _{B}n}\circ \chi _{\left(
B^{\wedge _{E}2}\right) ^{\otimes _{B}n-1},B^{\wedge _{E}2}} \\
&=&p_{n}^{A}\circ m_{A}\circ \left[ \left( \overline{m}_{A}^{n-2}\circ
\left( i_{B^{\wedge _{E}2}}^{A}\right) ^{\otimes _{B}n-1}\right) \otimes
i_{B^{\wedge _{E}2}}^{A}\right] \\
&=&p_{n}^{A}\circ m_{A}\circ \left[ \left( i_{B^{\wedge _{E}n}}^{A}\circ
\varphi _{n-1}\right) \otimes i_{B^{\wedge _{E}2}}^{A}\right]
=p_{n}^{A}\circ m_{A}\circ \left( i_{B^{\wedge _{E}n}}^{A}\otimes
i_{B^{\wedge _{E}2}}^{A}\right) \circ \left( \varphi _{n-1}\otimes B^{\wedge
_{E}2}\right) \\
&&\overset{\text{(\ref{form: coro grAlg1})}}{=}p_{n}^{A}\circ i_{B^{\wedge
_{E}n+1}}^{A}\circ m_{n-1,1}^{A}\circ \left( \varphi _{n-1}\otimes B^{\wedge
_{E}2}\right) =m_{\wedge }^{n-1,1}\circ \left( \varphi _{n-1}\otimes
B^{\wedge _{E}2}\right) .
\end{eqnarray*}%
Hence (\ref{form: phi}) holds. Next we prove%
\begin{equation}
i_{B^{\wedge _{E}n+1}}^{E}\circ \varphi _{n}=\overline{m}_{E}^{n-1}\circ
\left( i_{B^{\wedge _{E}2}}^{E}\right) ^{\otimes _{B}n},\text{ for every }%
n\geq 2.  \label{form: phi 2}
\end{equation}%
This is achieved by induction, composing on the right both sides with the
epimorphism $\chi _{\left( B^{\wedge _{E}2}\right) ^{\otimes
_{B}n-1},B^{\wedge _{E}2}}$ and using (\ref{form: phi}), (\ref{form:
Sweed7,5}).

Now, if $\varphi _{n}$ is an epimorphism, then
\begin{equation*}
\left( B^{\wedge _{E}n+1},i_{B^{\wedge _{E}n+1}}^{E}\right) =\mathrm{Im}%
\left[ \overline{m}_{E}^{n-1}\circ \left( i_{B^{\wedge _{E}2}}^{E}\right)
^{\otimes _{B}n}\right] =\left( B^{\wedge _{E}2}\right) ^{\cdot _{E}n}.
\end{equation*}%
Conversely, if $\left( B^{\wedge _{E}n+1},i_{B^{\wedge _{E}n+1}}^{E}\right)
=\left( B^{\wedge _{E}2}\right) ^{\cdot _{E}n},$ then there exists an
epimorphism%
\begin{equation*}
\varphi _{n}^{\prime }:\left( B^{\wedge _{E}2}\right) ^{\otimes
_{B}n}\rightarrow B^{\wedge _{E}n+1}
\end{equation*}%
such that
\begin{equation*}
i_{B^{\wedge _{E}n+1}}^{E}\circ \varphi _{n}^{\prime }=\overline{m}%
_{E}^{n-1}\circ \left( i_{B^{\wedge _{E}2}}^{E}\right) ^{\otimes _{B}n}.
\end{equation*}%
In view of (\ref{form: phi 2}) and since $i_{B^{\wedge _{E}n+1}}^{E}$ is a
monomorphism, we get that $\varphi _{n}=\varphi _{n}^{\prime }$.
\end{proof}

\begin{corollary}
\label{coro: pre-Lift}Let $H$ be a subbialgebra of a bialgebra $E$ over a
field $K$. The following assertions are equivalent.

\begin{enumerate}
\item[$\left( 1\right) $] $gr_{H}E$ is the bialgebra of type one associated
to $H$ and $\frac{H\wedge _{E}H}{H}$.

\item[$\left( 2\right) $] $gr_{H}E$ is strongly $%
\mathbb{N}
$-graded as an algebra i.e. $gr_{H}E$ is generated as an algebra by $H$ and $%
\frac{H\wedge _{E}H}{H}.$

\item[$\left( 3\right) $] $\oplus _{n\in \mathbb{N}}H^{\wedge _{E}n+1}$ is
strongly $%
\mathbb{N}
$-graded as an algebra.

\item[$\left( 4\right) $] $H^{\wedge _{E}n+1}=\left( H\wedge _{E}H\right)
^{\cdot _{E}n}$ for every $n\geq 2.$
\end{enumerate}
\end{corollary}

\begin{proof}
We apply Theorem \ref{teo: typeOne Alg} to the case $\left( \mathcal{M}%
,c\right) =\left( \mathfrak{Vec}\left( K\right) ,\tau \right) $ where $\tau $
is the canonical flip.
\end{proof}

\begin{remark}
Let $H$ be a subbialgebra of a bialgebra $E$ over a field $K$ and assume
that $H$ contains the coradical of $E$ (e.g. $E$ is connected). Assume that
one of the conditions of Corollary \ref{coro: pre-Lift} holds. Since $H$
contains the coradical of $E,$ then the filtration $\left( H^{\wedge
_{E}n+1}\right) _{n\in
\mathbb{N}
}$ is exhaustive hence $\left( 4\right) $ implies that $E$ is generated as
an algebra by $H\wedge _{E}H.$ The converse of this implication seems not to be true in general. Nevertheless we could not find a counterexample.

When $E$ is connected, in Corollary \ref{coro: pre-Lift}, we recover part of
\cite[Theorem 3.5]{Kharchenko} although for ordinary Hopf algebras.
\end{remark}

\section{The Associated Graded Algebra of a Bialgebra With Respect to a
Quotient bialgebra\label{sec: ass Grad alg Bialg}}

\begin{lemma}
\label{lem: graded braided Co}Let $\left( \mathcal{M},c\right) $ be a
cocomplete coabelian braided monoidal category such that the tensor product
commutes with direct sums. Consider in $\left( \mathcal{M},c\right) $ a
datum $(B,m_{B},u_{B},\Delta _{B},\varepsilon _{B})$ consisting of a graded
object $B$ with graduation defined by $\left( B_{k}\right) _{k\in \mathbb{N}%
} $ such that, with respect to this graduation,

\begin{itemize}
\item $\left( B,m_{B},u_{B}\right) $ is a graded algebra and

\item $(B,\Delta _{B},\varepsilon _{B})$ is a graded coalgebra.
\end{itemize}

Assume that for every $a,b\in
\mathbb{N}
,$%
\begin{eqnarray*}
&&\sum\limits_{s+t=a+b}\Delta _{a,b}m_{s,t}\left( p_{s}\otimes p_{t}\right)
\\
&=&\sum\limits_{s^{\prime }+t^{\prime }=a}\sum\limits_{s^{\prime \prime
}+t^{\prime \prime }=b}\left( m_{s^{\prime },t^{\prime }}\otimes
m_{s^{\prime \prime },t^{\prime \prime }}\right) \circ \left( B_{s^{\prime
}}\otimes c_{B_{s^{\prime \prime }},B_{t^{\prime }}}\otimes B_{t^{\prime
\prime }}\right) \circ \left( \Delta _{s^{\prime },s^{\prime \prime
}}\otimes \Delta _{t^{\prime },t^{\prime \prime }}\right) \circ \left(
p_{s^{\prime }+s^{\prime \prime }}\otimes p_{t^{\prime }+t^{\prime \prime
}}\right)
\end{eqnarray*}%
\begin{equation*}
\Delta _{0,0}\circ u_{0}=\left( u_{0}\otimes u_{0}\right) \circ \Delta _{%
\mathbf{1}}.
\end{equation*}%
Then $B$ is a graded braided bialgebra in $\left( \mathcal{M},c\right) $.
\end{lemma}

\begin{proof}
It is analogous to that of Lemma \ref{lem: graded braided}.
\end{proof}

\begin{notations}
From now on the following assumptions and notations will be used.

$\left( \mathcal{M},c\right) $ is a cocomplete abelian coabelian braided
monoidal category. Assume that the tensor product commutes with direct sums.
\newline
Let $\pi :E\rightarrow B$ be an epimorphism in $\mathcal{M}$ which is a
braided bialgebra homomorphism in $\mathcal{M}$ and let $\left(
I,i_{I}^{E}\right) :=\ker \left( \pi \right) .$ Assume that%
\begin{equation*}
\frac{E}{i_{I^{a+1}}^{I^{a}}}:\frac{E}{I^{a+1}}\rightarrow \frac{E}{I^{a}}
\end{equation*}%
is a split epimorphism for every $a\in
\mathbb{N}
$, where $i_{I^{a+1}}^{I^{a}}:I^{a+1}\rightarrow I^{a}$ is the canonical
injection.

The family $((\frac{E}{I^{a}})_{a\in \mathbb{N}},(\frac{E}{%
i_{I^{a+1}}^{I^{a}}})_{a\in \mathbb{N}})$ fulfills the conditions of Theorem %
\ref{pro: Sweedler} when regarded inside the dual of the abelian monoidal
category $\mathcal{M}.$ Thus we have the following exact sequence%
\begin{equation*}
\bigoplus\limits_{a+b=n}I^{a}\otimes I^{b}\overset{\nabla \left[ \left(
i_{I^{a}}^{E}\otimes i_{I^{b}}^{E}\right) _{a+b=n}\right] }{\longrightarrow }%
E\otimes E\overset{\Delta \left[ \left( p_{I^{a}}^{E}\otimes
p_{I^{b}}^{E}\right) _{a+b=n+1}\right] }{\longrightarrow }%
\bigoplus\limits_{a+b=n+1}\frac{E}{I^{a}}\otimes \frac{E}{I^{b}}.
\end{equation*}%
Let%
\begin{equation*}
\left( I_{n}\left( E\right) ,\beta _{n}\right) :=\mathrm{Im}\left\{ \Delta %
\left[ \left( p_{I^{a}}^{E}\otimes p_{I^{b}}^{E}\right) _{a+b=n+1}\right]
\right\}
\end{equation*}%
and let
\begin{equation*}
\gamma _{n}:E\otimes E\twoheadrightarrow I_{n}\left( E\right)
\end{equation*}
be the unique morphism such that
\begin{equation}
\beta _{n}\circ \gamma _{n}=\Delta \left[ \left( p_{I^{a}}^{E}\otimes
p_{I^{b}}^{E}\right) _{a+b=n+1}\right] .  \label{form: def gammaCo}
\end{equation}%
Since $\left( I^{a+b},i_{I^{a+b}}^{E}\right) =\mathrm{\mathrm{Im}}\left(
m_{E}\circ \left( i_{I^{a}}^{E}\otimes i_{I^{b}}^{E}\right) \right) ,$ we
have $p_{I^{a+b}}^{E}\circ m_{E}\circ \left( i_{I^{a}}^{E}\otimes
i_{I^{b}}^{E}\right) =0$ so that%
\begin{equation*}
p_{I^{n}}^{E}\circ m_{E}\circ \nabla \left[ \left( i_{I^{a}}^{E}\otimes
i_{I^{b}}^{E}\right) _{a+b=n}\right] =0
\end{equation*}%
and hence, by the exactness of the sequence above, there exists a unique
morphism
\begin{equation*}
\alpha _{n}:I_{n}\left( E\right) \rightarrow \frac{E}{I^{n}}
\end{equation*}%
such that%
\begin{equation}
p_{I^{n}}^{E}\circ m_{E}=\alpha _{n}\circ \gamma _{n},\text{ for every }n\in
\mathbb{N}
\text{.}  \label{form: def alphaCo}
\end{equation}
\end{notations}

\begin{lemma}
\label{lem: mLambdaCo}0) For every $s,t,u,v\in
\mathbb{N}
,$ we have
\begin{eqnarray}
&&\left( \beta _{u}\otimes \beta _{v}\right) \circ \left( \gamma _{u}\otimes
\gamma _{v}\right) \circ \left( E\otimes c\otimes E\right) \circ \left(
\Delta _{E}\otimes \Delta _{E}\right) \circ \left( i_{I^{s}}^{E}\otimes
i_{I^{t}}^{E}\right)  \label{form: Sweed5,5Co} \\
&=&\Delta \left[ \left( \left( \frac{E}{I^{a}}\otimes c_{\frac{E}{I^{b}},%
\frac{E}{I^{c}}}\otimes \frac{E}{I^{d}}\right) \circ \left( \left[ \left(
p_{I^{a}}^{E}\otimes p_{I^{b}}^{E}\right) \Delta _{E}i_{I^{s}}^{E}\right]
\otimes \left[ \left( p_{I^{c}}^{E}\otimes p_{I^{d}}^{E}\right) \Delta
_{E}i_{I^{t}}^{E}\right] \right) \right) _{\substack{ a+c=u+1  \\ b+d=v+1}}%
\right]  \notag
\end{eqnarray}

1) The following relations hold.%
\begin{equation}
\left( p_{I^{u}}^{E}\otimes p_{I^{v}}^{E}\right) \circ \Delta _{E}\circ
i_{I^{u+v-1}}^{E}=0\text{ for every }u,v\in
\mathbb{N}
,u+v\geq 1.  \label{form: Sweed6,5Co}
\end{equation}%
2) For every $a,b\in
\mathbb{N}
$, there exists a unique morphism
\begin{equation*}
\Delta _{\vee }^{a,b}:\frac{E}{I^{a+b+1}}\rightarrow \frac{E}{I^{a+1}}%
\otimes \frac{E}{I^{b+1}}
\end{equation*}%
such that
\begin{equation}
\Delta _{\vee }^{a,b}\circ p_{I^{a+b+1}}^{E}=\left( p_{I^{a+1}}^{E}\otimes
p_{I^{b+1}}^{E}\right) \Delta _{E},\text{ for every }a,b\in
\mathbb{N}
.  \label{form: Sweed7,5Co}
\end{equation}%
3) For every $a,b,c,d\in
\mathbb{N}
$, we have%
\begin{gather*}
\left( \Delta _{\vee }^{a,b}\otimes \frac{E}{I^{c+1}}\right) \circ \Delta
_{\vee }^{a+b,c}=\left( \frac{E}{I^{a+1}}\otimes \Delta _{\vee
}^{b,c}\right) \circ \Delta _{\vee }^{a,b+c}\text{,} \\
\left( \frac{E}{I^{d+1}}\otimes \varepsilon _{B}\right) \circ \Delta _{\vee
}^{d,0}=r_{E/I^{d+1}}^{-1},\qquad \left( \varepsilon _{B}\otimes \frac{E}{%
I^{d+1}}\right) \circ \Delta _{\vee }^{0,d}=l_{E/I^{d+1}}^{-1}\text{,} \\
\left( \frac{E}{i_{I^{a+2}}^{I^{a+1}}}\otimes \frac{E}{I^{b+1}}\right) \circ
\Delta _{\vee }^{a+1,b}=\Delta _{\vee }^{a,b}\circ \frac{E}{%
i_{I^{a+b+2}}^{I^{a+b+1}}},\qquad \left( \frac{E}{I^{a+1}}\otimes \frac{E}{%
i_{I^{b+2}}^{I^{b+1}}}\right) \circ \Delta _{\vee }^{a,b+1}=\Delta _{\vee
}^{a,b}\circ \frac{E}{i_{I^{a+b+2}}^{I^{a+b+1}}}
\end{gather*}
\end{lemma}

\begin{proof}
It is analogous to that of Lemma \ref{lem: mLambda}.
\end{proof}

\begin{proposition}
\label{pro: grE algCo}$\oplus _{n\in \mathbb{N}}\frac{E}{I^{n+1}}$ is a
graded coalgebra and there is a unique coalgebra structure on $%
gr_{I}E=\oplus _{n\in \mathbb{N}}gr_{I}^{n}E$ such that
\begin{equation*}
\oplus _{n\in \mathbb{N}}\frac{i_{I^{n}}^{E}}{I^{n+1}}:gr_{I}E\rightarrow
\oplus _{n\in \mathbb{N}}\frac{E}{I^{n+1}}
\end{equation*}%
is a coalgebra homomorphism and

\begin{enumerate}
\item $gr_{I}E=\oplus _{n\in \mathbb{N}}gr_{I}^{n}E$ is a graded coalgebra
such that $\oplus _{n\in \mathbb{N}}\frac{i_{I^{n}}^{E}}{I^{n+1}}$ is a
graded homomorphism;

\item
\begin{equation}
\left( \frac{i_{I^{a}}^{E}}{I^{a+1}}\otimes \frac{i_{I^{b}}^{E}}{I^{b+1}}%
\right) \circ \Delta _{a,b}^{gr_{I}E}=\Delta _{\vee }^{a,b}\circ \frac{%
i_{I^{a+b}}^{E}}{I^{a+b+1}};  \label{form: Sweed8Co}
\end{equation}

\item $\varepsilon _{gr_{I}E}=\varepsilon _{B}\circ p_{0}^{gr_{I}E}.$
\end{enumerate}
\end{proposition}

\begin{proof}
By Lemma \ref{lem: mLambdaCo}, we can apply the Lemma \ref{lem: GrDirect
coalg Inverse} to the family $(\frac{E}{I^{n+1}})_{n\in
\mathbb{N}
}.$
\end{proof}

\begin{theorem}
\label{teo: gr_IE Bialg}Let $\left( \mathcal{M},c\right) $ be a cocomplete
and complete abelian coabelian braided monoidal category satisfying $AB5$.
Assume that the tensor product commutes with direct sums.

Let $\pi :E\rightarrow B$ be an epimorphism in $\mathcal{M}$ which is a
braided bialgebra homomorphism in $\mathcal{M}$ and let $\left(
I,i_{I}^{E}\right) :=\ker \left( \pi \right) .$ Assume that%
\begin{equation*}
\frac{E}{i_{I^{a+1}}^{I^{a}}}:\frac{E}{I^{a+1}}\rightarrow \frac{E}{I^{a}}
\end{equation*}%
is a split epimorphism for every $a\in
\mathbb{N}
$, where $i_{I^{a+1}}^{I^{a}}:I^{a+1}\rightarrow I^{a}$ is the canonical
injection.\newline
Then $gr_{I}E$ is a graded braided bialgebra in $\left( \mathcal{M},c\right)
$.
\end{theorem}

\begin{proof}
It is analogous to that of Theorem \ref{teo: gr_BE Bialg}.
\end{proof}

\begin{theorem}
\label{teo: typeOne CoAlg}Let $\left( \mathcal{M},c\right) $ be a cocomplete
and complete abelian coabelian braided monoidal category satisfying $AB5$.
Assume that the tensor product commutes with direct sums. Let $\pi
:E\rightarrow B$ be an epimorphism in $\mathcal{M}$ which is a braided
bialgebra homomorphism in $\mathcal{M}$ and let $\left( I,i_{I}^{E}\right)
:=\ker \left( \pi \right) .$ Assume that%
\begin{equation*}
\frac{E}{i_{I^{a+1}}^{I^{a}}}:\frac{E}{I^{a+1}}\rightarrow \frac{E}{I^{a}}
\end{equation*}%
is a split epimorphism for every $a\in
\mathbb{N}
$, where $i_{I^{a+1}}^{I^{a}}:I^{a+1}\rightarrow I^{a}$ is the canonical
injection. The following assertions are equivalent.

\begin{enumerate}
\item[$\left( 1\right) $] $gr_{I}E$ is the braided bialgebra of type one
associated to $B=\frac{E}{I}$ and $\frac{I}{I^{2}}$.

\item[$\left( 2\right) $] $gr_{I}E$ is strongly $%
\mathbb{N}
$-graded as a coalgebra.

\item[$\left( 3\right) $] $C=\oplus _{n\in \mathbb{N}}\frac{E}{I^{n+1}}$ is
strongly $%
\mathbb{N}
$-graded as a coalgebra.

\item[$\left( 4\right) $] $\left( I^{n+1},i_{I^{n+1}}^{E}\right) =\left(
I^{2}\right) ^{\wedge _{E}n}$ for every $n\geq 2.$
\end{enumerate}
\end{theorem}

\begin{proof}
Consider the graded coalgebra homomorphism
\begin{equation*}
\oplus _{n\in \mathbb{N}}\frac{i_{I^{n}}^{E}}{I^{n+1}}:gr_{I}E\rightarrow
\oplus _{n\in \mathbb{N}}\frac{E}{I^{n+1}}
\end{equation*}%
of Proposition \ref{pro: grE algCo}. Note that in view of $AB5$ condition,
this morphism is indeed a monomorphism.

$\left( 1\right) \Leftrightarrow \left( 2\right) $ It follows in view of
\cite[Theorem 6.8]{AM} (where $AB5$ is required) and by Theorem \ref{teo: gr}%
.

$\left( 2\right) \Leftrightarrow \left( 3\right) $ It follows by Theorem \ref%
{teo: strongly Alg co} that, by Lemma \ref{lem: mLambdaCo}, can be applied
to the family $(\frac{E}{I^{n+1}})_{n\in
\mathbb{N}
}.$

$\left( 3\right) \Leftrightarrow \left( 4\right) $ Let $\psi :\oplus _{n\in
\mathbb{N}}\frac{E}{I^{n+1}}\rightarrow T^{c}=T_{\frac{E}{I}}^{c}\left(
\frac{E}{I^{2}}\right) $ be the canonical morphism arising from the
universal property of the tensor algebra and let
\begin{equation*}
\psi _{n}:\frac{E}{I^{n+1}}\rightarrow \left( \frac{E}{I^{2}}\right)
^{\square _{B}n}
\end{equation*}%
be its graded $n$-th component. In view of \cite[Theorem 2.22]{AM}, $\left(
3\right) $ is equivalent to require that $\psi _{n}$ is an epimorphism for
every $n\geq 2$ (note that $\psi _{0}$ and $\psi _{1}$ are always
isomorphisms). Let us prove that%
\begin{equation}
\zeta _{\left( \frac{E}{I^{2}}\right) ^{\square _{B}n-1},\frac{E}{I^{2}}%
}\circ \psi _{n}=\left( \psi _{n-1}\otimes \frac{E}{I^{2}}\right) \circ
\Delta _{\vee }^{n-1,1},\text{ for every }n\geq 2.  \label{form: psi}
\end{equation}%
where $\zeta _{X,Y}:X\square _{B}Y\rightarrow X\otimes Y$ denotes the
canonical injection.

Note that, being $\psi $ a graded homomorphism and in view of \cite[Theorem
2.16 and Proposition 2.19]{AM}, one has%
\begin{equation*}
\psi _{n}\circ p_{n-1}^{C}=p_{n-1}^{T^{c}}\circ \psi =\left(
p_{1}^{C}\right) ^{\square _{B}n-1}\circ \overline{\Delta }_{C}^{n-2}
\end{equation*}%
so that%
\begin{eqnarray*}
&&\zeta _{\left( \frac{E}{I^{2}}\right) ^{\square _{B}n-1},\frac{E}{I^{2}}%
}\circ \psi _{n} \\
&=&\zeta _{\left( \frac{E}{I^{2}}\right) ^{\square _{B}n-1},\frac{E}{I^{2}}%
}\circ \left( p_{1}^{C}\right) ^{\square _{B}n}\circ \overline{\Delta }%
_{C}^{n-1}\circ i_{n}^{C}=\left[ \left( p_{1}^{C}\right) ^{\square
_{B}n-1}\circ \overline{\Delta }_{C}^{n-2}\otimes p_{1}^{C}\right] \circ
\Delta _{C}\circ i_{n}^{C} \\
&=&\left[ \left( \psi _{n-1}\circ p_{n-1}^{C}\right) \otimes p_{1}^{C}\right]
\circ \Delta _{C}\circ i_{n}^{C}=\left( \psi _{n-1}\otimes C_{1}\right)
\circ \left( p_{n-1}^{C}\otimes p_{1}^{C}\right) \circ \Delta _{C}\circ
i_{n}^{C} \\
&&\overset{\text{(\ref{form: coro grCoalg1})}}{=}\left( \psi _{n-1}\otimes
\frac{E}{I^{2}}\right) \circ \Delta _{n-1,1}^{C}\circ p_{n}^{C}\circ
i_{n}^{C}=\left( \psi _{n-1}\otimes \frac{E}{I^{2}}\right) \circ \Delta
_{\vee }^{n-1,1}.
\end{eqnarray*}%
Hence (\ref{form: psi}) holds. Let us prove by induction that%
\begin{equation}
\psi _{n}\circ p_{I^{n+1}}^{E}=\left( p_{I^{2}}^{E}\right) ^{\square
_{B}n}\circ \overline{\Delta }_{E}^{n-1},\text{ for every }n\geq 2.
\label{form: psi 2}
\end{equation}%
$n=2)$ We have%
\begin{eqnarray*}
&&\zeta _{\frac{E}{I^{2}},\frac{E}{I^{2}}}\circ \psi _{2}\circ p_{I^{3}}^{E}%
\overset{\text{(\ref{form: psi})}}{=}\left( \psi _{1}\otimes \frac{E}{I^{2}}%
\right) \circ \Delta _{\vee }^{1,1}\circ p_{I^{3}}^{E} \\
&=&\Delta _{\vee }^{1,1}\circ p_{I^{3}}^{E}\overset{\text{(\ref{form:
Sweed7,5Co})}}{=}\left( p_{I^{2}}^{E}\otimes p_{I^{2}}^{E}\right) \circ
\Delta _{E}=\zeta _{\frac{E}{I^{2}},\frac{E}{I^{2}}}\circ \left(
p_{I^{2}}^{E}\square _{B}p_{I^{2}}^{E}\right) \circ \overline{\Delta }_{E}
\end{eqnarray*}

$n-1\Rightarrow n)$ We have%
\begin{eqnarray*}
&&\zeta _{\left( \frac{E}{I^{2}}\right) ^{\square _{B}n-1},\frac{E}{I^{2}}%
}\circ \psi _{n}\circ p_{I^{n+1}}^{E}\overset{\text{(\ref{form: psi})}}{=}%
\left( \psi _{n-1}\otimes \frac{E}{I^{2}}\right) \circ \Delta _{\vee
}^{n-1,1}\circ p_{I^{n+1}}^{E} \\
&&\overset{\text{(\ref{form: Sweed7,5Co})}}{=}\left( \psi _{n-1}\otimes
\frac{E}{I^{2}}\right) \circ \left( p_{I^{n}}^{E}\otimes
p_{I^{2}}^{E}\right) \circ \Delta _{E} \\
&=&\left( \left( p_{I^{2}}^{E}\right) ^{\square _{B}n-1}\otimes
p_{I^{2}}^{E}\right) \circ \left( \overline{\Delta }_{E}^{n-2}\otimes \frac{E%
}{I^{2}}\right) \circ \Delta _{E}=\zeta _{\left( \frac{E}{I^{2}}\right)
^{\square _{B}n-1},\frac{E}{I^{2}}}\circ \left( p_{I^{2}}^{E}\right)
^{\square _{B}n}\circ \overline{\Delta }_{E}^{n-1}.
\end{eqnarray*}

We have so proved that (\ref{form: psi 2}) holds.

Now, if $\psi _{n}$ is a monomorphism, then
\begin{equation*}
\left( I^{n+1},i_{I^{n+1}}^{E}\right) =\ker \left( \psi _{n}\circ
p_{I^{n+1}}^{E}\right) =\mathrm{\ker }\left[ \left( p_{I^{2}}^{E}\right)
^{\square _{B}n}\circ \overline{\Delta }_{E}^{n-1}\right] =\mathrm{\ker }%
\left[ \left( p_{I^{2}}^{E}\right) ^{\otimes n}\circ \Delta _{E}^{n-1}\right]
=\left( I^{2}\right) ^{\wedge _{E}n}.
\end{equation*}%
Conversely, if $\left( I^{n+1},i_{I^{n+1}}^{E}\right) =\left( I^{2}\right)
^{\wedge _{E}n},$ then $\left( p_{I^{2}}^{E}\right) ^{\square _{B}n}\circ
\overline{\Delta }_{E}^{n-1}$ factors to a monomorphism%
\begin{equation*}
\psi _{n}^{\prime }:\frac{E}{I^{n+1}}\rightarrow \left( \frac{E}{I^{2}}%
\right) ^{\square _{B}n}
\end{equation*}%
such that $\psi _{n}^{\prime }\circ p_{I^{n+1}}^{E}=\left(
p_{I^{2}}^{E}\right) ^{\square _{B}n}\circ \overline{\Delta }_{E}^{n-1}.$ In
view of (\ref{form: psi 2}) and since $p_{I^{n+1}}^{E}$ is an epimorphism,
we get that $\psi _{n}=\psi _{n}^{\prime }$.
\end{proof}

\begin{corollary}
Let $J$ an ideal of a bialgebra $E$ over a field $K$ and assume that $J$ is
also a coideal. The following assertions are equivalent.

\begin{enumerate}
\item[$\left( 1\right) $] $gr_{J}E$ is the bialgebra of type one associated
to $\frac{E}{J}$ and $\frac{J}{J^{2}}$.

\item[$\left( 2\right) $] $gr_{J}E$ is strongly $%
\mathbb{N}
$-graded as a coalgebra.

\item[$\left( 3\right) $] $\oplus _{n\in \mathbb{N}}\frac{E}{J^{n+1}}$ is
strongly $%
\mathbb{N}
$-graded as a coalgebra.

\item[$\left( 4\right) $] $J^{n+1}=\left( J^{2}\right) ^{\wedge _{E}n}$ for
every $n\geq 2.$
\end{enumerate}
\end{corollary}

\begin{proof}
We apply Theorem \ref{teo: typeOne CoAlg} to the case $\left( \mathcal{M}%
,c\right) =\left( \mathfrak{Vec}\left( K\right) ,\tau \right) $ where $\tau $
is the canonical flip.
\end{proof}

\appendix

\section{Technicalities}

\label{sec: techn}

\begin{proof}[Proof of Lemma \protect\ref{lem: mLambda}]
. Set $i_{n}:=i_{B^{\wedge _{E}a}}^{E}$ and $p_{n}:=p_{B^{\wedge
_{E}n}}^{E}. $

0) It follows by using (\ref{form: def gamma}) and naturality of $c.$

1) Let us prove by induction on $u+v\geq 1$ that $p_{u+v-1}m_{E}\left(
i_{u}\otimes i_{v}\right) =0.$

$u+v=1)$ is trivial as $i_{0}=0$.

$u+v>1)$ If $u=0$ or $v=0$ there is nothing to prove. Let $u,v>0$. Assume
that the statement is true for every $i,j$ such that $1\leq i+j<u+v$ and let
us prove it for $u+v.$ First of all we will prove that%
\begin{equation}
\left( p_{u-1}\otimes p_{v}\right) \circ \Delta _{E}\circ m_{E}\circ \left(
i_{u}\otimes i_{v}\right) =0.  \label{form:
step1}
\end{equation}%
Using the compatibility of $\Delta _{E}$ and $m_{E},$ and (\ref{form: def
alpha}) we get%
\begin{eqnarray*}
&&\left( p_{u-1}\otimes p_{v}\right) \circ \Delta _{E}\circ m_{E}\circ
\left( i_{u}\otimes i_{v}\right) \\
&=&\left( p_{u-1}\otimes p_{v}\right) \circ \left( m_{E}\otimes m_{E}\right)
\circ \left( E\otimes c\otimes E\right) \circ \left( \beta _{u}\otimes \beta
_{v}\right) \circ \left( \alpha _{u}\otimes \alpha _{v}\right)
\end{eqnarray*}%
Let us prove that
\begin{equation*}
\left( p_{u-1}\otimes p_{v}\right) \circ \left( m_{E}\otimes m_{E}\right)
\circ \left( E\otimes c\otimes E\right) \circ \left( \beta _{u}\otimes \beta
_{v}\right) =0.
\end{equation*}%
Since $\gamma _{u}\otimes \gamma _{v}$ is an epimorphism, this is equivalent
to prove that%
\begin{equation*}
\left( p_{u-1}\otimes p_{v}\right) \circ \left( m_{E}\otimes m_{E}\right)
\circ \left( E\otimes c\otimes E\right) \circ \left( \beta _{u}\otimes \beta
_{v}\right) \circ \left( \gamma _{u}\otimes \gamma _{v}\right) =0.
\end{equation*}%
We have%
\begin{eqnarray*}
&&\left( p_{u-1}\otimes p_{v}\right) \circ \left( m_{E}\otimes m_{E}\right)
\circ \left( E\otimes c\otimes E\right) \circ \left( \beta _{u}\otimes \beta
_{v}\right) \circ \left( \gamma _{u}\otimes \gamma _{v}\right) \\
&&\overset{\text{(\ref{form: Sweed5,5})}}{=}\nabla \left[ \left( \left( %
\left[ p_{u-1}m_{E}\left( i_{a}\otimes i_{b}\right) \right] \otimes \left[
p_{v}m_{E}\left( i_{c}\otimes i_{d}\right) \right] \right) \circ \left(
B^{\wedge _{E}a}\otimes c_{B^{\wedge _{E}c},B^{\wedge _{E}b}}\otimes
B^{\wedge _{E}d}\right) \right) _{\substack{ a+c=u+1  \\ b+d=v+1}}\right]
\end{eqnarray*}%
Note that $\left( a+b\right) +\left( c+d\right) =\left( a+c\right) +\left(
b+d\right) =u+v+2.$

If $u-1\geq a+b-1$ then $a+b\leq u<u+v$ and we have%
\begin{equation*}
p_{u-1}m_{E}\left( i_{a}\otimes i_{b}\right) =\frac{E}{\xi _{a+b-1}^{u-1}}%
p_{a+b-1}m_{E}\left( i_{a}\otimes i_{b}\right) =0.
\end{equation*}%
If $u-1<a+b-1,$ and $v<c+d-1,$ then $u+v\leq a+b+c+d-3=u+v-1.$ A
contradiction.

Then $v\geq c+d-1.$ Thus $c+d<v\leq u+v$ so that, as above, we get $%
p_{v}m_{E}\left( i_{c}\otimes i_{d}\right) =0.$

Hence%
\begin{equation*}
\left( p_{u-1}\otimes p_{v}\right) \circ \left( m_{E}\otimes m_{E}\right)
\circ \left( E\otimes c\otimes E\right) \circ \left( \beta _{u}\otimes \beta
_{v}\right) \circ \left( \gamma _{u}\otimes \gamma _{v}\right) =0
\end{equation*}%
and so (\ref{form: step1}) holds.

Let $\Delta _{u-1,v}:=\Delta _{B^{\wedge _{E}u-1},B^{\wedge _{E}v}}=\left(
p_{u-1}\otimes p_{v}\right) \circ \Delta _{E}.$ Then, as seen in \ref{claim:
wedge}, there exists a unique morphism
\begin{equation*}
\overline{\Delta }_{u-1,v}:\frac{E}{B^{\wedge _{E}u+v-1}}=\frac{E}{B^{\wedge
_{E}u-1}\wedge _{E}B^{\wedge _{E}v}}\rightarrow \frac{E}{B^{\wedge _{E}u-1}}%
\otimes \frac{E}{B^{\wedge _{E}v}}
\end{equation*}%
such that $\overline{\Delta }_{u-1,v}\circ p_{u+v-1}=\Delta _{u-1,v}.$
Furthermore $\overline{\Delta }_{u-1,v}$ is a monomorphism. From%
\begin{equation*}
\overline{\Delta }_{u-1,v}\circ \left[ p_{u+v-1}\circ m_{E}\circ \left(
i_{u}\otimes i_{v}\right) \right] =\Delta _{u-1,v}\circ m_{E}\circ \left(
i_{u}\otimes i_{v}\right) \overset{\text{(\ref{form: step1})}}{=}0
\end{equation*}%
we deduce $p_{u+v-1}\circ m_{E}\circ \left( i_{u}\otimes i_{v}\right) =0$ so
that we have proved (\ref{form: Sweed6,5}).

2) From (\ref{form: Sweed6,5}), we get $p_{a+b+1}m_{E}\left( i_{a+1}\otimes
i_{b+1}\right) =0$ for every $a,b\in
\mathbb{N}
.$ By the universal property of the kernel there exists a unique morphism $%
m_{\wedge }^{a,b}:B^{\wedge _{E}a+1}\otimes B^{\wedge _{E}b+1}\rightarrow
B^{\wedge _{E}a+b+1}$ such that (\ref{form: Sweed7,5}) holds.

3) From (\ref{form: Sweed7,5}), we get%
\begin{eqnarray*}
i_{a+b+c+1}\circ m_{\wedge }^{a+b,c}\circ \left( m_{\wedge }^{a,b}\otimes
B^{\wedge _{E}c+1}\right) &=&m_{E}\left( m_{E}\otimes E\right) \left(
i_{a+1}\otimes i_{b+1}\otimes i_{c+1}\right) \\
i_{a+b+c+1}\circ m_{\wedge }^{a,b+c}\circ \left( B^{\wedge _{E}a+1}\otimes
m_{\wedge }^{b,c}\right) &=&m_{E}\left( E\otimes m_{E}\right) \left(
i_{a+1}\otimes i_{b+1}\otimes i_{c+1}\right)
\end{eqnarray*}%
so that, by associativity of $m_{E},$ we obtain
\begin{equation*}
i_{a+b+c+1}\circ m_{\wedge }^{a+b,c}\circ \left( m_{\wedge }^{a,b}\otimes
B^{\wedge _{E}c+1}\right) =i_{a+b+c+1}\circ m_{\wedge }^{a,b+c}\circ \left(
B^{\wedge _{E}a+1}\otimes m_{\wedge }^{b,c}\right) .
\end{equation*}%
Since $i_{a+b+c+1}$ is a monomorphism, we deduce that%
\begin{equation*}
m_{\wedge }^{a+b,c}\circ \left( m_{\wedge }^{a,b}\otimes B^{\wedge
_{E}c+1}\right) =m_{\wedge }^{a,b+c}\circ \left( B^{\wedge _{E}a+1}\otimes
m_{\wedge }^{b,c}\right) .
\end{equation*}%
On the other hand, by applying (\ref{form: Sweed7,5}), we infer that%
\begin{equation*}
i_{d+1}\circ m_{\wedge }^{d,0}\circ \left( B^{\wedge _{E}d+1}\otimes
u_{B}\right) =m_{E}\left( E\otimes u_{E}\right) \left( i_{d+1}\otimes
\mathbf{1}\right) =r_{E}\left( i_{d+1}\otimes \mathbf{1}\right)
=i_{d+1}\circ r_{B^{\wedge _{E}d+1}}
\end{equation*}%
so that, since $i_{d+1}$ is a monomorphism, we obtain $m_{\wedge
}^{d,0}\circ \left( B^{\wedge _{E}d+1}\otimes u_{B}\right) =r_{B^{\wedge
_{E}d+1}}.$ Similarly we prove that $m_{\wedge }^{0,d}\circ \left(
u_{B}\otimes B^{\wedge _{E}d+1}\right) =l_{B^{\wedge _{E}d+1}}$. We have%
\begin{eqnarray*}
&&i_{a+b+2}\circ m_{\wedge }^{a+1,b}\circ \left( i_{B^{\wedge
_{E}a+1}}^{B^{\wedge _{E}a+2}}\otimes B^{\wedge _{E}b+1}\right) =m_{E}\circ
\left( i_{a+2}\otimes i_{b+1}\right) \circ \left( i_{B^{\wedge
_{E}a+1}}^{B^{\wedge _{E}a+2}}\otimes B^{\wedge _{E}b+1}\right) \\
&=&m_{E}\circ \left( i_{a+1}\otimes i_{b+1}\right) =i_{a+b+1}\circ m_{\wedge
}^{a,b}=i_{a+b+2}\circ i_{B^{\wedge _{E}a+b+1}}^{B^{\wedge _{E}a+b+2}}\circ
m_{\wedge }^{a,b}.
\end{eqnarray*}%
Since $i_{a+b+2}$ is a monomorphism, we deduce that $m_{\wedge
}^{a+1,b}\circ \left( i_{B^{\wedge _{E}a+1}}^{B^{\wedge _{E}a+2}}\otimes
B^{\wedge _{E}b+1}\right) =i_{B^{\wedge _{E}a+b+1}}^{B^{\wedge
_{E}a+b+2}}\circ m_{\wedge }^{a,b}.$ The right hand version of this formula
follows by similar arguments.
\end{proof}

\begin{lemma}
There exists a unique morphism%
\begin{equation*}
\theta _{a,b}:\sum_{u+v=a+b+2}B^{\wedge _{E}u}\otimes B^{\wedge
_{E}v}\rightarrow \frac{B^{\wedge _{E}a+1}}{B^{\wedge _{E}a}}\otimes \frac{%
B^{\wedge _{E}b+1}}{B^{\wedge _{E}b}}
\end{equation*}%
such that%
\begin{equation}
\left( p_{B^{\wedge _{E}a}}^{B^{\wedge _{E}a+1}}\otimes p_{B^{\wedge
_{E}b}}^{B^{\wedge _{E}b+1}}\right) \circ \beta _{a+b+1}=\left( \frac{%
i_{B^{\wedge _{E}a+1}}^{E}}{B^{\wedge _{E}a}}\otimes \frac{i_{B^{\wedge
_{E}b+1}}^{E}}{B^{\wedge _{E}b}}\right) \circ \theta _{a,b}.
\label{form: teta1}
\end{equation}%
Moreover, for every $a,b\in
\mathbb{N}
$, we have%
\begin{equation}
\theta _{a,b}\circ \alpha _{a+b+1}=\Delta _{a,b}^{gr_{B}E}\circ p_{B^{\wedge
_{E}a+b}}^{B^{\wedge _{E}a+b+1}},  \label{form: def theta}
\end{equation}%
\begin{equation}
\theta _{a,b}\circ \gamma _{a+b+1}=\nabla \left[ \left( \delta
_{u,a+1}\delta _{v,b+1}p_{B^{\wedge _{E}a}}^{B^{\wedge _{E}a+1}}\otimes
p_{B^{\wedge _{E}b}}^{B^{\wedge _{E}b+1}}\right) _{u+v=a+b+2}\right] .
\label{form: teta2}
\end{equation}
\end{lemma}

\begin{proof}
Set $i_{n}:=i_{B^{\wedge _{E}a}}^{E}$, $p_{n}:=p_{B^{\wedge _{E}n}}^{E}$ and
$p_{n}^{n+1}:=p_{B^{\wedge _{E}n}}^{B^{\wedge _{E}n+1}}.$

Since $p_{s}i_{t}=0$ for every $s\geq t$, we have%
\begin{equation*}
\left( \frac{E}{B^{\wedge _{E}a}}\otimes \frac{E}{\xi _{b}^{b+1}}\right)
\circ \left( p_{a}\otimes p_{b}\right) \circ \beta _{a+b+1}\circ \gamma
_{a+b+1}\overset{\text{(\ref{form: def gamma})}}{=}\nabla \left[ \left[
\left( p_{a}i_{u}\otimes p_{b+1}i_{v}\right) \right] _{u+v=a+b+2}\right] =0.
\end{equation*}%
Since $\gamma _{a+b+1}$ is an epimorphism, we get%
\begin{equation}
\left( \frac{E}{B^{\wedge _{E}a}}\otimes \frac{E}{\xi _{b}^{b+1}}\right)
\circ \left( p_{a}\otimes p_{b}\right) \circ \beta _{a+b+1}=0.
\label{form: mah1}
\end{equation}%
By the universal property of the kernel applied to the exact sequence,%
\begin{equation*}
0\rightarrow \frac{E}{B^{\wedge _{E}a}}\otimes \frac{B^{\wedge _{E}b+1}}{%
B^{\wedge _{E}b}}\overset{\frac{E}{B^{\wedge _{E}a}}\otimes \frac{i_{b+1}}{%
B^{\wedge _{E}b}}}{\longrightarrow }\frac{E}{B^{\wedge _{E}a}}\otimes \frac{E%
}{B^{\wedge _{E}b}}\overset{\frac{E}{B^{\wedge _{E}a}}\otimes \frac{E}{\xi
_{b}^{b+1}}}{\longrightarrow }\frac{E}{B^{\wedge _{E}a}}\otimes \frac{E}{%
B^{\wedge _{E}b+1}}\rightarrow 0.
\end{equation*}%
there exists a unique morphism
\begin{equation*}
\theta _{a,b}^{\prime }:\sum_{u+v=a+b+2}B^{\wedge _{E}u}\otimes B^{\wedge
_{E}v}\rightarrow \frac{E}{B^{\wedge _{E}a}}\otimes \frac{B^{\wedge _{E}b+1}%
}{B^{\wedge _{E}b}}
\end{equation*}%
such that%
\begin{equation*}
\left( \frac{E}{B^{\wedge _{E}a}}\otimes \frac{i_{b+1}}{B^{\wedge _{E}b}}%
\right) \circ \theta _{a,b}^{\prime }=\left( p_{a}\otimes p_{b}\right) \circ
\beta _{a+b+1}.
\end{equation*}%
Now%
\begin{equation*}
\left( \frac{E}{B^{\wedge _{E}a+1}}\otimes \frac{i_{b+1}}{B^{\wedge _{E}b}}%
\right) \circ \left( \frac{E}{\xi _{a}^{a+1}}\otimes \frac{B^{\wedge _{E}b+1}%
}{B^{\wedge _{E}b}}\right) \circ \theta _{a,b}^{\prime }=\left( \frac{E}{\xi
_{a}^{a+1}}\otimes \frac{E}{B^{\wedge _{E}b}}\right) \circ \left(
p_{a}\otimes p_{b}\right) \circ \beta _{a+b+1}=0
\end{equation*}%
where the last equality follows analogously to (\ref{form: mah1}). Since $%
\frac{E}{B^{\wedge _{E}a+1}}\otimes \frac{i_{b+1}}{B^{\wedge _{E}b}}$ is a
monomorphism we get $\left( \frac{E}{\xi _{a}^{a+1}}\otimes \frac{B^{\wedge
_{E}b+1}}{B^{\wedge _{E}b}}\right) \circ \theta _{a,b}^{\prime }=0.$ By the
universal property of the kernel applied to the exact sequence,%
\begin{equation*}
0\rightarrow \frac{B^{\wedge _{E}a+1}}{B^{\wedge _{E}a}}\otimes \frac{%
B^{\wedge _{E}b+1}}{B^{\wedge _{E}b}}\overset{\frac{i_{a+1}}{B^{\wedge _{E}a}%
}\otimes \frac{B^{\wedge _{E}b+1}}{B^{\wedge _{E}b}}}{\longrightarrow }\frac{%
E}{B^{\wedge _{E}a}}\otimes \frac{B^{\wedge _{E}b+1}}{B^{\wedge _{E}b}}%
\overset{\frac{E}{\xi _{a}^{a+1}}\otimes \frac{B^{\wedge _{E}b+1}}{B^{\wedge
_{E}b}}}{\longrightarrow }\frac{E}{B^{\wedge _{E}a+1}}\otimes \frac{%
B^{\wedge _{E}b+1}}{B^{\wedge _{E}b}}.
\end{equation*}%
there exists a unique morphism
\begin{equation*}
\theta _{a,b}:\sum_{u+v=a+b+2}B^{\wedge _{E}u}\otimes B^{\wedge
_{E}v}\rightarrow \frac{B^{\wedge _{E}a+1}}{B^{\wedge _{E}a}}\otimes \frac{%
B^{\wedge _{E}b+1}}{B^{\wedge _{E}b}}
\end{equation*}%
such that%
\begin{equation*}
\left( \frac{i_{a+1}}{B^{\wedge _{E}a}}\otimes \frac{B^{\wedge _{E}b+1}}{%
B^{\wedge _{E}b}}\right) \circ \theta _{a,b}=\theta _{a,b}^{\prime }.
\end{equation*}%
Thus (\ref{form: teta1}) holds true.

Let us prove (\ref{form: def theta}). We have%
\begin{equation*}
\left( \frac{i_{a+1}}{B^{\wedge _{E}a}}\otimes \frac{i_{b+1}}{B^{\wedge
_{E}b}}\right) \circ \theta _{a,b}\circ \alpha _{a+b+1}=\left( p_{a}\otimes
p_{b}\right) \circ \beta _{a+b+1}\circ \alpha _{a+b+1}=\left( p_{a}\otimes
p_{b}\right) \circ \Delta _{E}\circ i_{a+b+1}.
\end{equation*}%
On the other hand, in view of definition of $\Delta _{a,b}^{B}$ (see \ref%
{claim: wedge}), we have%
\begin{eqnarray*}
&&\left( \frac{i_{a+1}}{B^{\wedge _{E}a}}\otimes \frac{i_{b+1}}{B^{\wedge
_{E}b}}\right) \circ \Delta _{a,b}^{gr_{B}E}\circ p_{a+b}^{a+b+1}\overset{%
\text{(\ref{form: Delta grCo})}}{=}\Delta _{a,b}^{B}\circ \frac{i_{a+b+1}}{%
B^{\wedge _{E}a+b}}\circ p_{a+b}^{a+b+1} \\
&=&\Delta _{a,b}^{B}\circ p_{a+b}\circ i_{a+b+1}=\left( p_{a}\otimes
p_{b}\right) \circ \Delta _{E}\circ i_{a+b+1}.
\end{eqnarray*}%
Hence%
\begin{equation*}
\left( \frac{i_{a+1}}{B^{\wedge _{E}a}}\otimes \frac{i_{b+1}}{B^{\wedge
_{E}b}}\right) \circ \theta _{a,b}\circ \alpha _{a+b+1}=\left( \frac{i_{a+1}%
}{B^{\wedge _{E}a}}\otimes \frac{i_{b+1}}{B^{\wedge _{E}b}}\right) \circ
\Delta _{a,b}^{gr_{B}E}\circ p_{a+b}^{a+b+1}.
\end{equation*}%
Since $\frac{i_{a+1}}{B^{\wedge _{E}a}}\otimes \frac{i_{b+1}}{B^{\wedge
_{E}b}}$ is a monomorphism, we get (\ref{form: def theta}).

Let us prove (\ref{form: teta2}). We have%
\begin{eqnarray*}
&&\left( \frac{i_{a+1}}{B^{\wedge _{E}a}}\otimes \frac{i_{b+1}}{B^{\wedge
_{E}b}}\right) \circ \theta _{a,b}\circ \gamma _{a+b+1}\overset{(\ref{form:
teta1})}{=}\left( p_{a}\otimes p_{b}\right) \circ \beta _{a+b+1}\circ \gamma
_{a+b+1} \\
&&\overset{\text{(\ref{form: def gamma})}}{=}\left( p_{a}\otimes
p_{b}\right) \circ \nabla \left[ \left( i_{u}\otimes i_{v}\right)
_{u+v=a+b+2}\right] =\nabla \left[ \left( p_{a}i_{u}\otimes
p_{b}i_{v}\right) _{u+v=a+b+2}\right] \\
&=&\nabla \left[ \delta _{u,a+1}\delta _{v,b+1}\left( p_{a}i_{a+1}\otimes
p_{b}i_{b+1}\right) _{u+v=a+b+2}\right] \\
&=&\left( \frac{i_{a+1}}{B^{\wedge _{E}a}}\otimes \frac{i_{b+1}}{B^{\wedge
_{E}b}}\right) \circ \nabla \left[ \delta _{u,a+1}\delta _{v,b+1}\left(
p_{a}^{a+1}\otimes p_{b}^{b+1}\right) _{u+v=a+b+2}\right] .
\end{eqnarray*}%
Since $\frac{i_{a+1}}{B^{\wedge _{E}a}}\otimes \frac{i_{b+1}}{B^{\wedge
_{E}b}}$ is a monomorphism we get (\ref{form: teta2}).
\end{proof}


\begin{thebibliography}{AMS1}
\bibitem[AM1]{AM} A. Ardizzoni and C. Menini, \emph{Braided Bialgebras of
Type One}, Comm. Algebra, Vol. \textbf{36}(11) (2008), 4296-4337.

\bibitem[AM2]{Connected} A. Ardizzoni and C. Menini, \emph{Some Remarks on
Connected Coalgebras}, Algebr. Represent. Theory, Vol. \textbf{12} (2009),
235-249.

\bibitem[AM3]{AM2} A. Ardizzoni and C. Menini, \emph{A Categorical Proof of
a Useful Result}, in "Modules and Comodules" Proceedings of a conference
dedicated to Robert Wisbauer. Edited by T. Brezi\'{n}ski, J. L. Gómez Pardo,
I. Shestakov and P. F. Smith. Trends in Math Vol. \textbf{XII}, Birkhäuser
Verlag, Basel, 2008, 31-45.

\bibitem[AMS1]{AMS:Cotensor} A. Ardizzoni, C. Menini and D. \c{S}tefan,
\emph{Cotensor Coalgebras in Monoidal Categories}, Comm. Algebra, Vol.
\textbf{35}, N. 1 (2007), 25--70.

\bibitem[AMS2]{AMS} A. Ardizzoni, C. Menini and D. \c{S}tefan, \emph{%
Hochschild Cohomology And 'Smoothness' In Monoidal Categories}, J. Pure
Appl. Algebra, \textbf{208} (2007), 297--330.

\bibitem[AS]{AS} N. Andruskiewitsch and H-J. Schneider, \emph{Pointed Hopf
algebras}. New directions in Hopf algebras, 1--68, Math. Sci. Res. Inst.
Publ., \textbf{43}, Cambridge Univ. Press, Cambridge, 2002.

\bibitem[Ka]{Kassel} C. Kassel, \emph{Quantum Groups, }Graduate Text in
Mathematics \textbf{155},\textbf{\ }Springer, 1995.

\bibitem[Kh]{Kharchenko} V. K. Kharchenko, \emph{Connected braided Hopf
algebras.} (English summary) J. Algebra \textbf{307} (2007), no. 1, 24--48.

\bibitem[Mj1]{Majid} S. Majid, \emph{Foundations of quantum group theory},
Cambridge University Press, 1995.

\bibitem[NT]{NT} C. N\u{a}st\u{a}sescu, B. Torrecillas, \emph{Graded
coalgebras}, Tsukuba J. Math. \textbf{17} (1993), 461--479.

\bibitem[Ni]{Ni} Nichols, W. D. Bialgebras of type one. Comm. Algebra
\textbf{6} (1978), no. 15, 1521--1552.

\bibitem[Po]{Po} N. Popescu, \emph{Abelian Categories with Application to
Rings and Modules}, Academic Press, London \& New York, (1973).

\bibitem[Ro]{Ro} M. Rosso, \emph{Quantum groups and quantum shuffles},
Invent. Math. \textbf{133} (1998), no. 2, 399--416.
\end{thebibliography}
\end{document}